\documentclass[12pt,english]{smfartss}%{demorgan}%{smfart}%{demorgan}
\title[A first-order heory is stable iff its type space is simplicially contractible]{%\selectlanguage{english}
	A first-order theory is stable iff its type space %of types over any set 
	is simplicially contractible } 
\thanks{We thank M.Bays for many helpful and useful discussions. The meaning in homotopy theory of the simplicial formula we use  was pointed out by V.Sosnilo.
Comments to be sent to either  \href{https://t.me/joinchat/GVRrKxbSO8EWehZYReTKeQ}{here} or
{\tt mi\!\!\!ishap\!\!\!p@sd\!\!\!df.org}. 
\\University of Haifa. This research was supported by ISF grant 290/19. 
These notes report on work and progress, check updates at \url{mishap.sdf.org/rfc22.pdf}}
\date{2013 %What do you gain by pretending so ?
 Die Mathematiker sind eine Art Franzosen:  Redet man zu ihnen, so
 \"ubersetzen sie es in ihre Sprache, und dann ist es alsobald ganz etwas
 anderes. --- J.W. von Goethe. %Maximen und Reflexionen.% Nr.~1005.%
 Aphorismen und Aufzeichnungen. Nach den Handschriften des Goethe- und Schiller-Archivs hg. von Max Hecker, Verlag der Goethe-Gesellschaft,
  Weimar 1907,
  Aus dem Nachlass, Nr.~1005, Uber Natur und Naturwissenschaft. Maximen und Reflexionen.\\ What do you gain by pretending so ? \\ {\tt mi\!\!\!ishap\!\!\!p@sd\!\!\!df.org}. }
\usepackage{MnSymbol}%amsmath}
\DeclareMathAlphabet{\mathbb}{U}{msb}{m}{n}
\DeclareFontFamily{U}{rcjhbltx}{}
\DeclareFontShape{U}{rcjhbltx}{m}{n}{<->rcjhbltx}{}
\DeclareSymbolFont{hebrewletters}{U}{rcjhbltx}{m}{n}

\let\aleph\relax\let\beth\relax
\let\gimel\relax\let\daleth\relax

\DeclareMathSymbol{\aleph}{\mathord}{hebrewletters}{39}
\DeclareMathSymbol{\beth}{\mathord}{hebrewletters}{98}
\DeclareMathSymbol{\gimel}{\mathord}{hebrewletters}{103}
\DeclareMathSymbol{\daleth}{\mathord}{hebrewletters}{100}

\DeclareMathSymbol{\lamed}{\mathord}{hebrewletters}{108}
\DeclareMathSymbol{\mem}{\mathord}{hebrewletters}{109}
\DeclareMathSymbol{\ayin}{\mathord}{hebrewletters}{96}
\DeclareMathSymbol{\tsadi}{\mathord}{hebrewletters}{118}
\DeclareMathSymbol{\qof}{\mathord}{hebrewletters}{113}
\DeclareMathSymbol{\shin}{\mathord}{hebrewletters}{152}
\usepackage{smfthm-treplo}
\usepackage{vmargin}
\usepackage{wasysym}
\usepackage{caption}
\usepackage{pdfpages}
\usepackage[matrix, arrow,all,cmtip,color]{xy}
\usepackage{url}
\def\includegraphics[#1]#2{    \pdfximage width  \linewidth {#2} %.25\hsize{man.png}
    \pdfrefximage\pdflastximage}
\newbox\TestBox
\def\Remove #1 {\setbox\TestBox=\hbox{#1}%
        \leavevmode\rlap{\vrule height 2.5pt depth-1.75pt width\wd\TestBox}%
	        \box\TestBox\ }
\newcommand{\bi}{\begin{itemize}}
\newcommand{\ei}{\end{itemize}}
\newcommand{\bd}{\begin{description}}
\newcommand{\ed}{\end{description}}

\def\lra{\longrightarrow}

\def\xra{\xrightarrow}

\def\sing{\operatorname{Sing}}

\usepackage{hyperref}
\usepackage{endnotes}

\def\rrt#1#2#3#4#5#6#7{\xymatrix{ {#1} \ar[r]^{} \ar@{->}[d]_{#2} & {#4} \ar[d]^{#5} \\ {#3}  \ar[r] \ar@{-->}[ur]^{#7}& {#6} }}

\def\lra{\longrightarrow}

\def\sSets{\operatorname{sSets}}

\def\Itp{\text{1-tp}}

\def\id{\operatorname{id}}

\def\llrra{\leftrightarrow}

\def\xra{\xrightarrow}

\def\Dop{\Delta^{\operatorname{op}}}
\def\Sets{\text{Sets}}
\def\Hom{\text{Hom}}

\def\PPhi{\ensuremath{\Filt}}
\newtheorem{propo}{Proposition}

\newtheorem{question}{Question}

\usepackage[ethiop,english]{babel}
\usepackage{ethiop}

 \newcommand{\ethi}{\selectlanguage{ethiop}}
\def\Filt{{\ethi\ethmath{wA}}}
\def\sFilt{s{{\ethi\ethmath{wA}}}}

\def\Sets{\text{Sets}}
\def\Hom{\text{Hom}}

\def\sPhi{\ensuremath{\sFilt}}

\def\PPhi{\ensuremath{\Filt}}

\def\PHi{\ethi\ethmath{wE}}
\def\sPHi{s\ethi\ethmath{wE}}
\def\skiip#1{}

\def\inv{^{-1}}

\def\bqqq{\begin{quote}}
\def\eqqq{\end{quote}}

\def\Aut{\operatorname{Aut}}
\def\Topp{\operatorname{Top}}

\def\tp{\op{tp}}

\def\op{\operatorname}

\def\aand{\op{\,\&\,}}
\def\SS{{\Bbb S}}
\def\TTop{\operatorname{\it Top}}
\def\sTop{\operatorname{\it sTop}}

\def\fSets{\operatorname{\it finiteSets}}

\def\sSets{\operatorname{\it sSets}}
\def\Cone{\operatorname{Cone}}

\def\HHom{\op{HHom}}

\def\nfSets{\operatorname{\it finiteSets}_{\neq\emptyset}}
\def\spfSets{\operatorname{\it sProFiniteSets}}
\def\sProFiniteSets{\spfSets}

\def\pr{\operatorname{pr}}
\def\const{\operatorname{const}} 
\def\tensor{\otimes} 

\def\Sing{\operatorname{Sing}}
\def\footnoteskipped#1{}
\def\CCC{\mathfrak C}
\def\BBB{\Bbb B}
\def\EEE{\Bbb E}

\def\BBB{\Bbb B} 
\def\AutCB{\Aut(\CCC/\!B)}
\def\Groups{\op{Groups}}
\def\Ssigma{\mathbb S}

\begin{document}
\selectlanguage{english} 
\catcode`\_=8\catcode`\^=7 \catcode`\_=8
\author{misha gavrilovich} %. University of Haifa.}
%erequest for comments (report on work-in-progress)}%research notes not for publication)}
\begin{abstract}
A definable type of a first-order theory is the same as 
a section (retraction) of the simplicial path space (decalage) of its space of types 
viewed as a simplicial topological space; as is well-known, in the category of simplicial sets
such sections correspond to homotopies contracting each connected component.
	Without the simplicial language this is stated in 
\cite[Exercise 8.3.3]{TZ},
%	although th`ey avoid simplicial language and state it as 
which defines 
a bijection between the set of all $1$-types definable over a parameter set $B$, 
and the set of 
all ``coherent'' families of continuous sections $\pi_n:S^T_n(B)\to S^T_{n+1}(B)$
where $S^T_n(B)$ is the Stone space of types with $n$ variables
of the theory $T$ with parameters in $B$.

%%%
%%%
%%%\cite[Exercise 8.3.3]{TZ} 
%%%%Show that this 
%%%defines, for any first order theory $T$, a
%%%bijection between the set of all 
%%%	types definable over a parameter set $B$ and the set of `all ``coherent'' families $(\pi_n )$
%%%	of continuous sections $\pi_n:S_n (B) \to S_{n+1} (B)$',
%%%	i.e.~in the simplicial language,
%%%        the set of all sections (retractions) 
%%%	$$\SS_\bullet(B)\xra{\pi_\bullet}
%%%	\SS_\bullet(B)\circ[+1] \xra{\pr_{1,2,3,...}} 
%%%\SS_\bullet(B)$$
%%%of the path-end coordinate projection of the simplicial path decalage object 
%%%	$\SS_\bullet(B)\circ[+1]\xra{\pr_{1,2,...}} \SS_\bullet(B)$ %:\Dop\lra \TTop$ 
%%%	%of the identity map $\SS_\bullet(B)\lra \SS_\bullet(B)$
%%%	where %the simplicial Stone type space 
%%%	$\SS_\bullet(B):n\longmapsto S_{n+1}(B)$ is the simplicial space of types over $B$ of theory $T$.
%%%        In the category of simplicial sets %$\sSets$ 
%%%	such a factorisation correspond to a homotopy contracting each connected component. 

Thus the definition of stability ``each type is definable'' 
	says that
	{\em a first order theory is stable iff %in $\sTop$ or in $\spfSets$ 
	its %simplicial 
	space of types is simplicially contractible}, % in $\sTop$ or in $\spfSets$. 
in the precise sense that
the simplicial type space functor $\SS^T_\bullet(B):\Delta^{op}\lra \Topp$, 
	$n\longmapsto \SS_{n+1}(B)$  fits %in the category of simplicial topological spaces %$\sTop$ 
	into a certain well-known simplicial diagram in the category of simplicial topological spaces
	which %in the category of simplicial sets %$\sSets$ 
	does define contractibility for fibrant simplicial sets. 
	%In fact, it is not yet clear how fair is it to say that this diagram defines 
	%``contractibility''.

	In this note we spell out this and similar diagrams representing notions in model theory
	such as a parameter set and a type,  a type being invariant, definable, and product of invariant types, 
	and give pointers to the same diagrams in homotopy theory.
\end{abstract}
\maketitle

%%%%	In simplicial language, \cite[Exercise 8.3.3]{TZ} says that a global type definable over a model 
%%%%	is a homotopy contracting the simplicial Stone space of types over the model.
%%%%	This implies that a theory is stable iff its simplicial Stone space of types is contractible, in some precise simplicial sense.
%%%%
%%%%	

In this note we rewrite several definitions in stability theory in terms of simplicial diagrams,
and give pointers to similar diagrams in homotopy theory. This note was started 
when we noted the simplicial language almost explicit %implicit (in fact rather explicit!)
in 
\cite[Exercise 8.3.3]{TZ}
which establishes a bijection between the set of all global types definable over a set $B$ 
and the set of all ``coherent'' families of continuous sections $\pi_n:\SS_n(B)\lra \SS_{n+1}(B)$, $n>0$,
where, as usual, $\SS_n(B)$ denotes the Stone space of $n$-types over $B$ 
of a theory $T$:
$$ \pi_n: r(y_1,...,y_n) \longmapsto \{\,\varphi(x,y_1,...,y_n) : \operatorname{d}_p \varphi \in r \}$$
If $p(x/\CCC)$ is a global type invariant over $B$, this map can be described 
in terms of product of types as $\pi_n: r(y_1,...,y_n) \longmapsto p(x)\tensor r(y_1,...,y_n)$

In the simplicial language, such a ``coherent'' family of continuous sections 
is precisely the lifting map in the following diagram in the category of simplicial topological spaces
or profinite sets, as explained in \ref{ExTrans}:
			\begin{equation}\begin{gathered}\label{SBlift0}\xymatrix{ & \SS_\bullet(B)\circ[+1]\ar[d]|{\operatorname{pr}_{2,3,..}} \\
			\SS_\bullet(B) \ar[r]|{\operatorname{id}} \ar@{-->}[ru]|{\pi_\bullet} & \SS_\bullet(B) }
			%\label{SBliftCAut}
			\xymatrix{ &  \mathfrak C_\bullet/\!\!\Aut(\mathfrak C/B) \circ[+1]\ar[d]|{\operatorname{pr}_{2,3,..}} \\
			\mathfrak C_\bullet/\!\!\Aut(\mathfrak C/B)  \ar[r]|{\operatorname{id}} \ar@{-->}[ru]|{\pi_\bullet} 
			& \mathfrak C_\bullet/\!\!\Aut(\mathfrak C/B) }\end{gathered} \end{equation}
Here $\SS_\bullet(B):n\longmapsto \SS_n(B)$ is the simplicial space of types of $T$
over a parameter set $B\subset \CCC$, 
and $[+1]:\Dop\lra\Dop, n\longmapsto n+1$ is the decalage shift endomorphism of $\Dop$ so that
$\SS_\bullet(B)\circ[+1]:n\longmapsto \SS_{n+1}(B)$ is what is called {\em the simplicial path space of $\SS_\bullet(B)$}.
The diagram on the left expands the diagram on the right:
%and %$\CCC$ is a monster model of $T$, 
$\CCC_\bullet:n\longmapsto \CCC^{n+1}$ is the simplicial set represented by the monster model $\CCC$ of $T$
where each $\CCC^{n+1}$ is equipped with the topology %pulled back from the type space 
%$\SS_{n+1}(B)=\CCC^{n+1}/\!\Aut(\CCC,B)$; 
generated by the solution sets $\{ (x_0,...,x_n)\in \CCC^{n+1} : \CCC\models \varphi(x_0,..,x_n,b_1,..,b_m),m\geq 0, b_1,...,b_m\in B \}$ of formulas with parameters in $B$;
the quotient is taken by the diagonal action.
Dropping the continuity requirement (i.e.~considering this diagram in $\sSets$) 
leads to the notion of a {\em global type invariant over $B$}; in the category $\sPhi$ of simplcial filters it defines a notion similar to non-forking.
Two ``coherent'' families $\pi_n^p,\pi_n^q:\SS_n(B)\lra \SS_{n+1}(B)$ of sections can be composed in an obvious way
$\SS_n(B) \xra {\pi_n} \SS_{n+1}(B) \xra{\pi_n^q} \SS_{n+2}(B)$, and the composition 
is a coherent family of sections corresponding to the product $p(x)\tensor q(x)$ of types. 
In simplicial terms (\S\ref{sProduct}) you say that 
given liftings $\pi_\bullet^p,\pi_\bullet^q:\SS_\bullet(B)\lra \SS_\bullet(B)\circ[+1]$, 
form the composition $\SS_\bullet(B) \xra {\pi_\bullet^p} \SS_\bullet(B)\circ [+1] \xra{ \pi_\bullet^q[+1]}\SS_\bullet(B) \circ[+2]$.
Hence, a Morley sequence of an invariant type $p$ %, or rather its type, 
is the infinite composition 
$$\SS_\bullet(B) \xra {\pi_\bullet^p} \SS_\bullet(B)\circ [+1] \xra{ \pi_\bullet^p[+1]}\SS_\bullet(B) \circ[+2]\xra{\pi_\bullet^p[+2]}...$$

In $\sSets$ for fibrant simplicial sets
these diagrams define the notion of a homotopy contracting each connected component,
and in $\TTop$ correspond to factorisations $X\lra \operatorname{Cone}(X)\lra X$
or $X\lra \Ssigma X \lra X$ of $\id:X\lra X$
though the cone or suspension\footnote{
Recall $\Cone(X):=X\times [0,1]/X\times \{1\}$, and 
$\Ssigma(X):=X\star\{-1,1\}=X\times [-1,1]/\{\,X\times\{-1\},\,%\approx\{-1\},\,\,
X\times\{1\}\,\}$. %\approx\{1\}\}$.
Also, $\Ssigma S^n=S^{n+1}$ where $S^n$ denotes the $n$-th sphere, 
$\pi_n(\Ssigma^k(X),x)=\pi_{n+k}(X,x)$, $0\leqslant k\leqslant n$,
and, more generally, $[\Ssigma X, Y]=[X,\Omega Y]$ where $[-,-]$ denote the homotopy classes of maps, and 
$\Omega Y:=\Hom(S^1,Y)$ is the loop space of $Y$.}
for a connected nice enough space $X$; for $X$ not connected one needs
to take the disjoint union of cones, resp.~suspensions, of the connected components of $X$.
%to take the cone or suspension of each connected component separately. 
%"Each 1-type over $B$ is definable" corresponds precisely to a simplicial diagram defining 
%"contractible". 
In the category $\sPhi$ of simplicial filters the same diagram captures the notion of convergence.

Recall that a theory is stable iff for any set (equiv., any model) 
 each $1$-type over the set is definable. 
Hence, 
%(\ref{SBlift0}) says, 
in a certain precise sense given by the diagram (\ref{SBlift33}) below, %that 
\begin{quote}{\em a theory is stable iff its space of types over any set  
	is simplicially contractible}, in the precise sense that it fits into the following diagram 
	in the category $\sTop$ or its full subcategory $\spfSets$
\end{quote}
\begin{equation}\begin{gathered}\label{SBlift33}
	\xymatrix{ &  \SS_\bullet(B)\circ[+1]\ar[d]|{\pr_1\times \operatorname{pr}_{2,3,..}} \\
\const_\bullet  \SS_1(B)\times 	\SS_\bullet(B) \ar[r]|{\operatorname{id}} \ar@{-->}[ru]|{\pi_\bullet} & 
\const_\bullet \SS_1(B) \times \SS_\bullet(B) }\end{gathered}\end{equation}
Here $\const_\bullet \SS_1(B)$ denotes the constant functor %$\Delta^{op}\lra \TTop$, 
$(\const_\bullet \SS_1(B))_n:=\SS_1(B) $. 

Unfortunately, it is not quite clear to us how fair is it to say that this diagram defines 
contractibility. Perhaps informally one may say that this diagram says that 
the space can be contracted to each of its points. 

Simplicially, a parameter set $A\subset \CCC$, or rather its complete diagram, 
resp.~an $1$-type over $A\subset \CCC$, can be described as a map in $\sSets$
from the simplicial set $|A|_\bullet:n\longmapsto A^{n+1}$ represented by $A$, 
to the space $\SS_\bullet(\emptyset)$ of types over the empty set,
resp.~to the decalage shifted space $\SS_\bullet(\emptyset)\circ[+1]$, see \S\ref{sParametersTypes}. %-\ref{sTypes}. 
	$$\xymatrix{ & \SS_\bullet(\emptyset)\circ [+1] \ar[d]|{pr_{2,3,..}} \\
	|A|_\bullet:= \Hom_{sets}(-,A) \ar[r]|(0.7){A\subset\CCC} \ar@{-->}[ur]|{p(x/A)}  & \SS_\bullet(\emptyset)
}\xymatrix{&  } 
\xymatrix@C=2.02cm{ |A|_\bullet:= \Hom_{sets}(-,A) \ar[d]\ar[r]|{p(x/A)}   & \SS_\bullet(\emptyset)\circ [+1] \ar[d]|{pr_{2,3,..}} \\
	\SS_\bullet(B)\ \ar[r]|(0.7){A\subset\CCC} \ar@{-->}[ur]|{}  & \SS_\bullet(\emptyset)}$$
\ \ \ \ \ \ \ \ \ \ \ \ \  \ \  a type over $A\subset\CCC$ \ \ \ \ \ \ \ \ \ \ \ \ \ \ \ \ \ \ \ \   \ \    \ \ \ \ \ \ \ \ \ a type over $A$ invariant over $B\subset A$
\subsubsection*{A little informal glossary of model theory vs topology} 
This leads to the following little very informal glossary of model theory vs topology:
notions on both sides fit into the same simplicial formulas.
Below $$\Cone_{c.c.}(X):=\bigsqcup\limits_{X_{c.c.}\text{ a connected component of }X} \Cone(X_{c.c.})$$
%=\bigsqcup\limits_{X_{c.c.}\text{ a connected component of }X} X_{c.c.}\times [0,1]/X\times \{1\}$$$%,
$$\Ssigma_{c.c.}(X):=\bigsqcup\limits_{X_{c.c.}\text{ a connected component of }X} \Ssigma(X_{c.c.})$$%
%= \bigsqcup\limits_{X_{c.c.}\text{ a connected component of }X} X_{c.c.}\times [-1,1]/\{X\times\{-1\},X\times\{1\}\}$$$
is the disjoint union of cones, resp.~suspensions, of the connected components of $X$.
This is well-defined in a useful way only for topological spaces ``nice'' enough. 

\begin{itemize}
\item stable theories --- contractible spaces
\item an invariant or definable type --- a homotopy 
	$\Cone_{c.c.}(X)\lra X$ or $h:\Ssigma_{c.c.} X \lra X$ contracting $\id:X\lra X$
\item product of invariant types $p(x)\tensor q(y)$ --- composition of homotopies 
	$$h^p\circ \Ssigma_{c.c.} h^q: \Ssigma_{c.c.}\Ssigma_{c.c.} X\lra X$$
\item a Morley sequence --- a sequence somewhat reminiscent of a specturum in stable homotopy theory
	$$h\circ \Ssigma_{c.c.} h \circ ...\circ\Ssigma_{c.c.}^{n-1}h: \Ssigma_{c.c.}^n X\lra X, n>0$$
\end{itemize}

Note that a standard advice (cf.~Remark~\ref{BGXasSB}) from homotopy theory would be to use a ``better'' 
standard construction of the quotient $\CCC_\bullet/\!\!\AutCB$, called
the classifying space or Borel construction of a group action.

Connection to stability theory arises if one considers the diagrams $(\ref{SBlift0}), (\ref{SBlift33})$
in the categories of simplicial topological spaces $\sTop$, profinite sets $\sProFiniteSets$, 
or filters $\sPhi$. 
The homotopy theory for simplicial topological spaces $\sTop$ and for simplicial 
profinite sets $\sProFiniteSets$  
is known; however, we were unable to understand how our diagrams 
relate to the notion of contracitibily/null-homotopy there.
In $\sPhi$ the same diagram defines the notion of convergence, but 
no theory of $\sPhi$ exists, only examples of reformulations of various basic
notions in topology, analysis, and model theory including a reformulation of stability as a lifting property \cite{Z1,Z2}. 

%%-%\subsubsection*{Purpose of this paper} This note is an early report on work-in-progress,
%%-%and its scope is limited to presenting an observation and a number of questions.
%%-%
%%-%
%%-%
%%-%report on work-in-progress %ese ``research notes not for publication'' 
%%-%is 
%%-%{\em request for comments and solicit collaborators} by presenting a precise mathematical observation, 
%%-%alongside requests for references, vague remarks and speculations. 
%%-%It is expected and hoped 
%%-%that an expert in either homotopy or model theory will be able to---and will!--suggest 
%%-%immediately improvements, lacking references, and questions easy to answer.
%%-%
%%-%
%%-%These standards (or lack of them) are stated explicitly for two reasons.
%%-%First, there is a tendency to view a written statement as ipso facto
%%-%authoritative, and the reader may be misled into wasting time reading this text
%%-%rather than waiting for more something more finished. 
%%-%Second, we hope to promote the exchange and discussion of
%%-%considerably less than authoritative ideas, and ease  % Second, there is 
%%-%a natural
%%-%hesitancy to publish something unpolished for the sole purpose of requesting comments
%%-%and collaboration \cite{RFC3}.
%%-%%, and we hope to ease this inhibition.
%%-%
\subsubsection*{Structure of the paper} 
In \S\ref{bookkeeping} we sketch how to view simplicially sets of parameters and 
types over parameters: a set of parameters (resp.~a type)  or rather its complete diagram,
in a model of a theory
is a morphism from a representable set to the (resp., shifted) simplicial Stone space of the theory.
Following \cite[Exercise 8.3.3]{TZ}, invariance and definability of types 
are then interpreted as lifting diagrams (retractions) in $\sSets$ and $\sTop$.

In \S\ref{sec:ExTrans} we repeat in more detail some of \S\ref{bookkeeping} and 
explain in detail how to view 
\cite[Exercise 8.3.3]{TZ} in the simplicial language.
Care is taken so that \S\ref{sec:ExTrans} can be read independently. 

In \S\ref{TopSimpHomotopy} we explain that the simplicial formula (\ref{SBlift0}) defines the usual notion
of  {\em a homotopy contracting each connected component}, %being contractible 
in the category of topological spaces
when applied to the singular complexes of sufficiently nice topological spaces.

Thus, in a certain precise sense, {\em a definable global type} is a homotopy contracting the simplicial Stone space of types. 
The product $p(x)\tensor q(x)$ of two global invariant types corresponds to a composition of such homotopie 
in $\sSets$, and thus the type of a Morley sequence
corresponds to iteratively composing in $\sSets$ an invariant type with itself shifted 
$...\pi_\bullet[n]\circ\pi_\bullet[n-1]\circ..\circ \pi_\bullet$.

Recall that a theory is stable iff each type over any set (equiv., any model)
is definable. Hence, (\ref{SBlift33}) in $\sTop$ says, in a certain precise sense, that 
\begin{quote}{\em a theory is stable iff its simplicial space of types over any parameter set $B$ 
	is contractible}
\end{quote}
In \S\ref{ConvergenceSimpHomotopy} we give a reference saying that the same formula 
defines the notion of convergence in the category of simplicial objects
of a category of filters. 

In \S\ref{TestProblems} we 
formulate hopefully easy problems which might be used to guide development of the simplicial reformulations
in model theory.

The reader may want to skip \S\ref{Questions} whose purpose is to formulate explicit requests 
for comments from readers, rather than be interesting in any way.
By including such a section, paraphrasing \cite{RFC3}, 
%%-%
%%-%\subsubsection*{Purpose of this paper} This note is an early report on work-in-progress,
%%-%and its scope is limited to presenting an observation and a number of questions.
%%-%
%%-%
%%-%
%%-%report on work-in-progress %ese ``research notes not for publication'' 
%%-%is 
%%-%{\em request for comments and solicit collaborators} by presenting a precise mathematical observation, 
%%-%alongside requests for references, vague remarks and speculations. 
%%-%It is expected and hoped 
%%-%that an expert in either homotopy or model theory will be able to---and will!--suggest 
%%-%immediately improvements, lacking references, and questions easy to answer.
%%-%
%%-%
%%-%These standards (or lack of them) are stated explicitly for two reasons.
%%-%First, there is a tendency to view a written statement as ipso facto
%%-%authoritative, and the reader may be misled into wasting time reading this text
%%-%rather than waiting for more something more finished. 
%%-%Second, 
we hope to promote the exchange and discussion of
considerably less than authoritative ideas, and ease  % Second, there is 
a natural
hesitancy to publish something unpolished for the sole purpose of requesting comments
and collaboration. % \cite{RFC3}.
%, and we hope to ease this inhibition.

\subsubsection*{Acknowledgements}
Will Johnson suggested looking at the simplicial sets of types, a suggestion I ignored even though I already had a half-baked characterisation 
of non-dividing using the simplicial Stone space in \cite{Z1}.
We thank David Blanc, Boris Chorny, Assaf Hasson, Kobi Peterzil, Ori Segel, and Andr\'es Villaveces for encouraging conversations.\\ 

\section{Simplicial language as bookkeeping names of the variables}\label{bookkeeping}

We demonstrate how to view parameters (rather, their complete diagrams) 
and types as morphisms in $\sSets$ or $\sTop$.

Essentially, {\em simplicial/functoriality is a way of bookkeeping  
	the names of variables or parameters} in a finitely consistent collection of formulas. 

\subsubsection*{Preliminaries: fixing simplicial notation}
Let $\Delta$ denote the category of non-empty finite linear orders 
denoted by $\{1<..<n\}$. Let $[+1]:\Delta\lra \Delta$ be the 
{\em decalage} endomorphism
adding a new least element to each finite liner order:
$$[+1]: \{1<..<n\} \longmapsto \{0<1<...<n\}$$
$$f:\{1<...<m\}\lra \{1<...<n\} \longmapsto f[+1](0):=0, f[+1](l)=:l$$ 
We denote the finite linear order $\{1<...<n\}$ either by $n^\leqslant$,
or by $[n-1]$, as is standard in simplicial literature. 
For a functor $X_\bullet:\Dop\lra \mathcal C$ in a category $\mathcal C$, 
inclusions $\{1<..<n\}\subset \{0<1<...<n\}$ induce maps 
$X_\bullet((n+1)^\leqslant) \lra X_\bullet(n^\leqslant)$, 
and these form a natural transformation we denote by
$\pr_{2,3,...}:X_\bullet\circ[+1]\lra X_\bullet$. 
Similarly, inclusions  $\{0\}\subset \{0<1<..<n\}$ 
induce maps 
$X_\bullet((n+1)^\leqslant) \lra X_\bullet(1^\leqslant)$, 
and these form a natural transformation we denote 
by 
$\pr_{1}:X_\bullet\circ[+1]\lra X(1^\leqslant)_\bullet$. 

{\em A simplicial object of a category $\mathcal C$} is by definition a functor $\Dop \lra \mathcal C$. 
They form a category usually denoted as $s\mathcal C$. We shall work with categories 
$\sTop$ of simplicial topological spaces, simplicial profinite sets $\spfSets$, and $\sSets$ of simplicial sets, and $\sPhi$ of simplicial filters (defined in \S\ref{sFdef}).

\subsection{Talking simplicially about parameters and types over them}
\label{sParametersTypes}

Fix a theory $T$ in a language $L$ and a ``monster'' 
model $\mathfrak C$ of $T$. Recall that ``monster'' here means
that we assume that $\mathfrak C$ is a model saturated and homogeneous with respect
to all ``small'' subsets; I think these assumptions imply (mean?)
that we can reconstruct $\CCC$ using the simplicial space $\CCC_\bullet/\Aut_L(\CCC/B)$ 
described below.   
For a subset $B\subset \mathfrak C$, 
let $\SS^T_n(B):=\mathfrak C^n / \Aut_L(\mathfrak C/B)$ denote 
the topological {\em Stone space of complete $n$-types over $B$}; 
we will often drop the superscript $T$. In model theory, orbits of $\Aut_L(\CCC/B)$ are referred to as {\em types}. 
Recall the topology on $\SS_n(B)$ %  
%that a subset $U\subset \SS_n(B)$ is open (and necessarily also closed) 
is generated by open (and necessarily also closed)
subsets $U_\phi=\{\,p(\bar x)\in S_n(B): \phi(\bar x)\in p(\bar x)\}$,
where $\phi(\bar x)$ varies though all the  
formulas in $L$ with parameters in $B$. % $\phi(\bar x)$
In other words, it is the weakest topology such that
each $L(B)$-formula defines a continuous  
$\Aut_L(\mathfrak C/B)$-invariant function $\SS_n(B)\lra \{0,1\}$
to the discrete two point set with the trivial action. 

As a set, the Stone space $\SS_n(B)=\CCC^n/\!\!\Aut_L(\CCC/B)$
is a quotient of $\CCC^n$ by $L$-autmorphisms fixing $B$ pointwise. 
We may equip $\CCC^n$ with a topology in an obvious way 
so that this equality holds in the category of topological spaces.

The spaces $\SS_n(B)$ form a functor $\SS_\bullet(B):\nfSets\lra \TTop$, 
$\{1,2,..,n\}\longmapsto \SS_n(B)$. It is a quotient of the %representable
functor $\CCC_\bullet := \Hom_\Sets(\{1,..,n\}, \CCC)= \CCC^n$
by $\Aut(\CCC/B)$ (however, note that here the topology on $\CCC^n$ 
is not the product topology).

As $\Dop\subset \fSets_{\neq\emptyset}$ is a subcategory, 
these functors restrict to $\Dop$, and thus can be considered
as objects of $\sTop$ and $\sSets$. In everything we say below 
about $\SS_\bullet(B)$, it does not matter which category the functor 
is defined on. 

\subsubsection{Parameters as morphisms to the simplicial Stone space}
%\begin{rema} 
For a set $A$, let $|A|_\bullet:=\Hom_\Sets(\{1,..,n\}, A)=|A|^n$
be the simplicial topological space represented by $A$;
here we equip $|A|^n$ with the discrete topology. 

Recall that {\em the complete diagram of a subset $A\subset M$ over parameters $B\subset M$} 
is the set of all %valid 
	formulas with parameters in $B$
and variables indexed by elements of $A$, 
which became valid after replacing the variables by the corresponding elements of $A$.

	To give a complete diagram of a subset $A\cup B\subset M$, $A\neq\emptyset$, 
	of a model $M$ of theory $T$ 
is the same as to give a simplicial map in $\sSets$
	$$|A|_\bullet:= \Hom_{\Sets}(-,A) \xra{\tau_A} 
	\SS^T_\bullet(B)$$
	such that for $b\in A\cap B$ $\tau_A(b)=\{x_b=b\}$.

	Indeed, for each $n$ we get a map $A^n \lra S_n(B)$, 
	i.e.~we know/specify the complete type of each tuple in $A$ over $B$.
	Functoriality ensures that these types are coherent, i.e.~$\tp(ab/B)$ 
	does extend $\tp(a/B)$ and $\tp(b/B)$.

	Thus, {\em simplicial/functoriality is a way to keep bookkeeping (track) of 
	the names of variables or parameters} in a type. 
%\end{rema}

\subsubsection{Types as morphisms to the shifted (decalage) simplicial Stone space}
%\begin{rema} %Recall that {\em the complete diagram of a subset $A\subset M$} is the set of all valid formulas with parameters in $A$.
	To give a complete $1$-type $p(x/AB)$ over  a subset $A\cup B\subset M$, $A\neq \emptyset$, 
	is the same as to give a lifting in $\sSets$ % simplicial map 
	$$\xymatrix{ & \SS_\bullet(B)\circ [+1] \ar[d]|{pr_{2,3,..}} \\
	|A|_\bullet:= \Hom_{sets}(-,A) \ar[r]|(0.7){\tau_A} \ar@{-->}[ur]|{p(x/AB)}  & \SS_\bullet(B)}$$ 
	Indeed, for each $n$ we get a map $A^n \lra \SS_{1+n}(B)$, 
	i.e.~we know/specify a type $p(-,\bar a/B)$ for each finite tuple $\bar a\subset A$.

	To give a complete $N$-type $p(\bar x/AB)$ over  a subset $A\cup B\subset M$, $A\neq\emptyset$,
	is the same as to give a simplicial map in $\sSets$ 
%	$$|A|_\bullet:= \Hom_{sets}(-,A) \lra \SS(T)_\bullet\circ [+N]$$ 
	$$\xymatrix{ & \SS_\bullet(B)\circ [+N] \ar[d]|{pr_{N+1, N+2,,..}} \\
	|A|_\bullet:= \Hom_{sets}(-,A) \ar[r]|(0.7){\tau_A} \ar@{-->}[ur]|{p(\bar x/AB)}  & \SS_\bullet(B)}$$ 
	Indeed, for each $n$ we get a map $A^n \lra S_{N+n}(T)$, i.e.~we know/specify a type $p(-,-,..,-,\bar a/B)$ for each finite tuple $\bar a\subset A$.

	Thus again we see that {\em simplicial/functoriality is a way 
	to keep bookkeeping (track) of the names of variables}...
%\end{rema}

\subsection{Invariant and definable types} 

\subsubsection{Invariant types} \cite[Exercise 8.3.3]{TZ} says, as we explain below in \S\ref{ExTrans}, that 
a global type $p(\bar x/\CCC)$ {\em invariant over $B$} is the same as the following lifting diagram in $\sSets$: %iff the following lifting property holds in $\sSets$  
	$$\xymatrix@C=+2.39cm{ %|A|_\bullet:= \Hom_{sets}(-,A) \ar[r]|{p(\bar x/AB)} \ar[d]  
		& \SS_\bullet(B)\circ [+N] \ar[d]|{pr_{N+1,N+2,,..}} \\
	\SS_\bullet(B)  \ar@{-->}[ur] \ar[r]|{\id} & \SS_\bullet(B)}$$

To see this, consider the diagram\footnote{A simplicially minded reader may consider this diagram as a definition of a global type invariant over $B$.}
in $\sSets$
\begin{equation}\label{defInv}%	$$
	\xymatrix@C=+2.39cm{ |\CCC|_\bullet:= \Hom_{sets}(-,\CCC) \ar[r]|{p(\bar x/AB)} \ar[d]|{AB\subset\CCC}  
		& \SS_\bullet(B)\circ [+N] \ar[d]|{pr_{N+1,N+2,..}} \\
\SS_\bullet(B)  \ar@{-->}[ur] \ar[r]|{\id} & \SS_\bullet(B)}\end{equation}%$$ 
The model $\CCC$ being saturated over $B$ means precisely that the map on the right is surjective at each level 
(i.e.~for each $n$ the map of $n$-simplicies is surjective). 
Hence, there is at most one lifting map, and we only need to check it is well-defined. 
The lifting map is well-defined iff 
for a tuple $\bar c\subset \CCC$, 
whether $\varphi(\bar x,\bar c)\in p(x/\CCC)$ depends only on the type $\tp(\bar c/B)$ of the parameters over $B$. 
This is precisely the definition of invariance over $B$.

It follows that a type $p(\bar x/AB)$ extends to a global type invariant over $B$ iff there is a lifting diagram in $\sSets$:
	$$\xymatrix@C=+2.39cm{ |A|_\bullet:= \Hom_{sets}(-,A) \ar[r]|{p(\bar x/AB)} \ar[d]|{AB\subset\CCC}  
		& \SS_\bullet(B)\circ [+N] \ar[d]|{pr_{N+1,N+2,..}} \\
	\SS_\bullet(B)  \ar@{-->}[ur] \ar[r]|{\id} & \SS_\bullet(B)}$$ 
%Indeed, the diagram says that for a tuple $\bar a\subset A$, 
%whether $\varphi(x,\bar a)\in p(x/AB)$ depends only on the type $\tp(\bar a/B)$ of the parameters over $B$. 

\subsubsection{Definable types} Further, \cite[Exercise 8.3.3]{TZ} says, as we explain below, that 
a global type $p(x/\CCC)$ {\em definable over $B$} is the diagonal map above is continuous, i.e.
it is the same as the following lifting diagram in $\sTop$ or, equiv., in $\spfSets$: %iff the lifting above is continuous, 
%i.e.~that the following lifting property holds in $\sTop$ or $\spfSets$ rather than $\sSets$:
	$$\xymatrix@C=+2.39cm{ %|A|_\bullet:= \Hom_{sets}(-,A) \ar[r]|{p(\bar x/AB)} \ar[d]  
		& \SS_\bullet(B)\circ [+N] \ar[d]|{pr_{N+1,N+2,..}} \\
	\SS_\bullet(B)  \ar@{-->}[ur] \ar[r]|{\id} & \SS_\bullet(B)}$$ 
To see this, consider the diagram (\ref{defInv}) in $\sTop$. 
%Indeed, the diagram says that for a tuple $\bar a\subset A$, 
%whether $\varphi(\bar x,\bar a)\in p(\bar x/AB)$ depends only on the type $\tp(\bar a/B)$ of the parameters over $B$. 
Continuity of the diagonal arrow says that for each  formula $\varphi(\bar x,\bar c)$ 
there is a formula $\operatorname d_p \varphi(\bar x,\bar c)$ such that for each tuple $\bar c \subset \CCC$ it holds 
$\varphi(\bar x,\bar c)\in p(\bar x/\CCC)$ iff $\operatorname d_p \varphi(\bar x,\bar c) \in \tp(\bar c/B)$.
This means precisely that 
$\operatorname d_p \varphi(\bar x,\bar a)$ is a $\varphi$-definition of $p$ over $B$.

It follows that a type $p(\bar x/AB)$ extends to a global type definable over $B$ iff there is a lifting diagram 
in $\sTop$ or, equiv.~its full subcategory $\spfSets$ of compact Hausdorff totally disconnected spaces
	$$\xymatrix@C=+2.39cm{ |A|_\bullet:= \Hom_{sets}(-,A) \ar[r]|{p(\bar x/AB)} \ar[d]|{AB\subset\CCC}  
		& \SS_\bullet(B)\circ [+N] \ar[d]|{pr_{N+1,N+2,..}} \\
	\SS_\bullet(B)  \ar@{-->}[ur] \ar[r]|{\id} & \SS_\bullet(B)}$$
%%%%
%%%%Indeed, the diagram says that for a tuple $\bar a\subset A$, 
%%%%whether $\varphi(\bar x,\bar a)\in p(\bar x/AB)$ depends only on the type $\tp(\bar a/B)$ of the parameters over $B$. 
%%%%Continuity of the diagonal arrow says that for each tuple $\bar a \subset A$ and a formula $\varphi(\bar x,\bar a)$ 
%%%%there is a formula $\operatorname d_p \varphi(\bar x,\bar a)$ such that 
%%%%$\varphi(\bar x,\bar a)\in p(\bar x/AB)$ iff $\operatorname d_p \varphi(\bar x,\bar a) \in \tp(\bar a/B)$.
%%%%This means precisely that 
%%%%$\operatorname d_p \varphi(\bar x,\bar a)$ is a $\varphi$-definition of $p$ over $B$.
%%%%
Because of its importance we rewrite the diagram expanding the notation for the type(=orbit) space:
	$$\xymatrix@C=+2.39cm{ |A|_\bullet:= \Hom_{sets}(-,A) \ar[r]|{p(\bar x/AB)} \ar[d]|{AB\subset\CCC}  
		& \CCC_\bullet/\!\Aut(\CCC,B)\circ [+N] \ar[d]|{pr_{N+1,N+2,..}} \\
	 \CCC_\bullet/\!\Aut(\CCC,B)  \ar@{-->}[ur] \ar[r]|{\id} &  \CCC_\bullet/\!\Aut(\CCC,B)}$$ 

\subsection{Various remarks} We make a couple of remarks.

\subsubsection{A notion of definable or invariant type ``better'' for homotopy theory}

A standard advice to improve the notion of type by a homotopy theorist, is to replace 
the quotient $\CCC/\!\!\Aut(\CCC/B))$ by a ``better'' and ''standard'' quotient 
which remembers more, the classifying space $\BBB(G,X)$ of a group $G:=\Aut(\CCC/B)$ acting on 
a space $X$. 

I have not yet tried to interpret it, and there are immediate technical difficulties 
with the way I explain it below.

\begin{rema}[$\BBB_\bullet(G,X)$ instead of $\SS_\bullet(B)$]\label{BGXasSB} 
For a group $G$ acting on a set $X$, 
there is an obvious canonical 
map $$X\times G\times ... \times G \lra X \times X \times...\times X$$
	$$(x,g_1,...,g_n)\longmapsto (x,xg\inv_1,...,xg\inv_n)$$
invariant under the diagonal action 
	$$(x,g_1,...,g_n)\longmapsto (xg,g_1g,...,g_ng).$$

	Recall\footnote{The following slightly paraphrased 
	quote from \href{https://ncatlab.org/nlab/show/simplicial+classifying+space}{[nlab,Borel construction]} helps intuition:
For $X$ a topological space, $G$ a topological group and $\rho:G\times X \lra X$
a continuous $G$-action (i.e.~a topological $G$-space), 
	the Borel construction of $\rho$ is the topological space $X \times_G \EEE G$, 
hence quotient of the product of $X$ 
	with the total space of the $G$-universal principal bundle $\EEE(G)$  
	by the diagonal action of $G$ on both.\\
\,\,\,	Analogously, for $G_\bullet:\Dop\lra\Groups$ a simplicial group, 
	$X_\bullet:\Dop\lra \Sets$ a simplicial set,  and 
	$G_\bullet  \times X_\bullet \lra X_\bullet$ 
        a simplicial group action, its Borel construction is 
	the quotient
% \frac { W \mathcal{G} \times X} {\mathcal{G}} \;\;\; \in \;\; SimplicialSets
	$( X_\bullet \times \EEE_\bullet(G) )	/ G_\bullet$
in $\sSets$
of the Cartesian product of $X_\bullet$ 
	with the universal principal simplicial complex $\EEE G_\bullet$ \
	by the diagonal action of $G_\bullet$ on these.
	}
	that the simplicial {\em Borel construction 
	$\BBB_\bullet(G,X)=X\times_G \EEE_\bullet(G)$ of a 
	group action $\rho:G\times X \lra X$} % $\BBB_\bullet(G,X)=X\times_G \EEE_\bullet(G)$} 
	is defined (explicitly given) by
	$$\BBB_\bullet(G,X)(n^\leqslant) :=(X\times G^n) /G $$ % \times ... \times G) / G$$
and this space is viewed as a ``better'', ``right'' quotient of $X$ by $G$. 

	The above gives rise to a simplicial map 
	$\BBB_\bullet(\CCC,\AutCB)\lra \SS_\bullet(B)\circ[+1]$
	from the classifying space $\BBB_\bullet(\CCC,\AutCB)$.
	Note, however, that its image contains only tuples with all elements realising the same type.
	
	Thus, a standard advice of a homotopy theorist would be to replace $\SS_\bullet(B)$
	in the diagrams above 
	by something related to $\BBB(\CCC,\AutCB)$. % in the diagrams above. 
	 In fact, perhaps it might be necessary to consider  $\BBB_\bullet(\CCC_\bullet,\AutCB_\bullet)$
	 for the simplicial group $\AutCB_\bullet$ where 
	 $\AutCB_\bullet(n^\leqslant):=\Aut(\CCC^n/\!B^n)$... 
\end{rema}
Does this advice make any sense ? I have not yet thought about it, and there are immediate technical difficulties ...

\subsubsection{Extending the parameter set of a type} A typical tool/problem in model theory 
is to extend/define {\em freely} a type to a larger parameter set. 
%TODO! This kinda 
This corresponds to finding/defining a canonical way to define liftings
	$$
%\begin{equation}
	\begin{gathered}
	\xymatrix{ & \SS_\bullet(B)\circ [+1] \ar[d]|{pr_{2,3,..}} \\
	|A|_\bullet \ar[r]|(0.4){A\subset\CCC} \ar@{-->}[ur]|{p(x/B)_{|A}?}  & \SS_\bullet(B)\\
		\text{  }& p(x/B)\text{ to }p(x/B)_{|A} }
		\xymatrix{ |A|_\bullet \ar[r]|(0.4){p(x/AB)} \ar[d] & \SS_\bullet(B)\circ [+1] \ar[d]|{pr_{2,3,..}} \\
	|A'|_\bullet \ar[r]|(0.4){A'\subset\CCC} \ar@{-->}[ur]|{p(x/AB)_{A'}?}  & \SS_\bullet(B)\\
		\text{ }&p(x/AB)\text{ to }p(x/AB)_{|A'}  }
\end{gathered}$$%\end{equation}
	Note that if the type $p(x/AB)$ extends to a global $B$-invariant type, 
	extending $p(x/AB)$ to $p(x/A'AB)$ is provided by taking the composition
$|A'|_\bullet \lra \SS_\bullet(B) \lra \SS_\bullet(B)\circ [+1]$. 
	$$
%\begin{equation}
	\begin{gathered}
		\xymatrix@C=+2.39cm{ & |A|_\bullet \ar[r]|(0.4){p(x/AB)} \ar[d] \ar[dl] & \SS_\bullet(B)\circ [+1] \ar[d]|{pr_{2,3,..}} \\
		|A'|_\bullet \ar[r]|(0.4){A'\subset\CCC} \ar@{-->}[urr]|(0.33){p(x/AB)_{A'}}  
		& \SS_\bullet(B)\ar[r]|{\id} \ar[ru]  & \SS_\bullet(B)\\\
		\text{ }&p(x/AB)\text{ to }p(x/AB)_{|A'}  }
\end{gathered}$$%\end{equation}

\subsection{Product of invariant types, and Morley sequences}\label{sProduct} 

\subsubsection{Product of invariant types}
	Note that liftings 
$\pi^p_\bullet,\pi^q_\bullet:\SS_\bullet(B)\lra \SS_\bullet(B)\circ[+1]$
can be composed as 
	%$\pi^p_\bullet[+1]\circ \pi^q_\bullet:\SS_\bullet(B)\lra \SS_\bullet(B)\circ[+2]$:
	$$\SS_\bullet(B)\xra{\pi^q_\bullet} \SS_\bullet(B)\circ[+1]\xra{\pi^p_\bullet[+1]}
	\SS_\bullet(B)\circ[+2]$$
	This construction corresponds to the product $p(x)\tensor q(y)$ of invariant types \cite[2.2.1]{Simon}. 
	The product of types is transitive, i.e.~$p(x)\tensor (q(y)\tensor s(z))=(p(x)\tensor q(y))\tensor s(z)$: this corresponds to 
	$$\pi^s_\bullet\circ (\pi^q_\bullet\circ \pi^p[+1])[+1]=
	\pi^s_\bullet\circ \pi^q_\bullet[+1]\circ \pi^p[+2]=
	(\pi^s_\bullet\circ \pi^q_\bullet[+1])\circ \pi^p[+2]$$
	$$\SS_\bullet(B)\xra{\pi^s_\bullet} \SS_\bullet(B)\circ[+1]\xra{\pi^p_\bullet[+1]}
			        \SS_\bullet(B)\circ[+2]
				\xra{\pi^p_\bullet[+2]}                                                        \SS_\bullet(B)\circ[+3]
				$$

%%%%Recall \cite[2.2.1]{Simon} 
%%%%that the type $p(x) \tensor p(y)$ can be defined by the following property.
%%%%
%%%%Given a formula $\varphi(x; y) \in L(B)$,
%%%%$A \subset B \subset \mathfrak C$, %we set 
%%%%$\varphi(x,y)\in p(x) \tensor q(y)$  iff $\varphi(x; b)\in p$ 
%%%%for some (equiv., any) $b \in \mathfrak C$  
%%%%with $b \models  q|B$.
%%%%

Recall \cite[2.2.1]{Simon} 
that {\em product  $p(x) \tensor p(y)$ of two $B$-invariant global types  
	$p(x/\CCC), q(y/\CCC) \in S(\mathfrak C)$} 
can be defined by the following property.

Given a formula $\varphi(x; y) \in L(C)$,
where $B \subset C \subset \mathfrak C$, %we set 
it holds 
$\varphi(x,y)\in p(x) \tensor q(y)$  iff $\varphi(x; c)\in p$ 
for some (equiv., any) $c \in \mathfrak C$  
	with $c \models  q_{|C}$ .

\subsubsection{Morley sequence} The $n$-type of a {\em Morley sequence of an invariant type $p(x)$} 
	is given by $p^{\tensor n }(x):=p(x_1)\tensor ...\tensor p(x_n)$, for $n>0$.  
	Thus a Morley sequence corresponds to taking the self-composition 
	$$\SS_\bullet(B)\xra{\pi^p_\bullet} \SS_\bullet(B)\circ[+1]
	\xra{\pi^p_\bullet[+1]}	\SS_\bullet(B)\circ[+2]\lra ... 
	\xra{\pi^p_\bullet[+n-1]} \SS_\bullet(B)\circ[+n]\lra... $$
	Somewhat more precisely,
{\em a Morley sequence of %a global $A$-invariant type 
of a global $B$-invariant type	$p(x/\mathfrak C)$ over $C$} %is a sequence indiscernible over $C$ 
	%whose EM-type is given by types corresponding to the sequence of sections
	is the restriction to $C$ of the following sequence of global $B$-invariant types:
	$$
p\in \SS_1(\CCC),\,\,\pi^p_\bullet(p)\in \SS_2(\CCC),\,\,\pi^p_\bullet[+1]\circ \pi^p_\bullet(p)\in \SS_3(\CCC)\,\,,...,
\,\,\pi^p_\bullet[n]\circ ...\circ \pi^p_\bullet(p)\in\SS_{n+1}(\CCC),...$$ 
Indiscernability of a Morley sequence follows from  associativity:
we have that for each %$\leqslant i_1\leqslant...\leqslant i_m\leqslant n$
map $[i_1\!<...<\!i_m]:m^\leqslant \lra n^\leqslant$ we have 
$$	\pi^p_\bullet[n]\circ ...\circ \pi^p_\bullet(p)[i_1\!<...<\!i_m]=\pi^p_\bullet[m]\circ ...\circ \pi^p_\bullet(p)$$

The reader may wish to compare this with a model theoretic exposition \cite[2.2.1]{Simon}:
\newline\noindent\includegraphics[test]{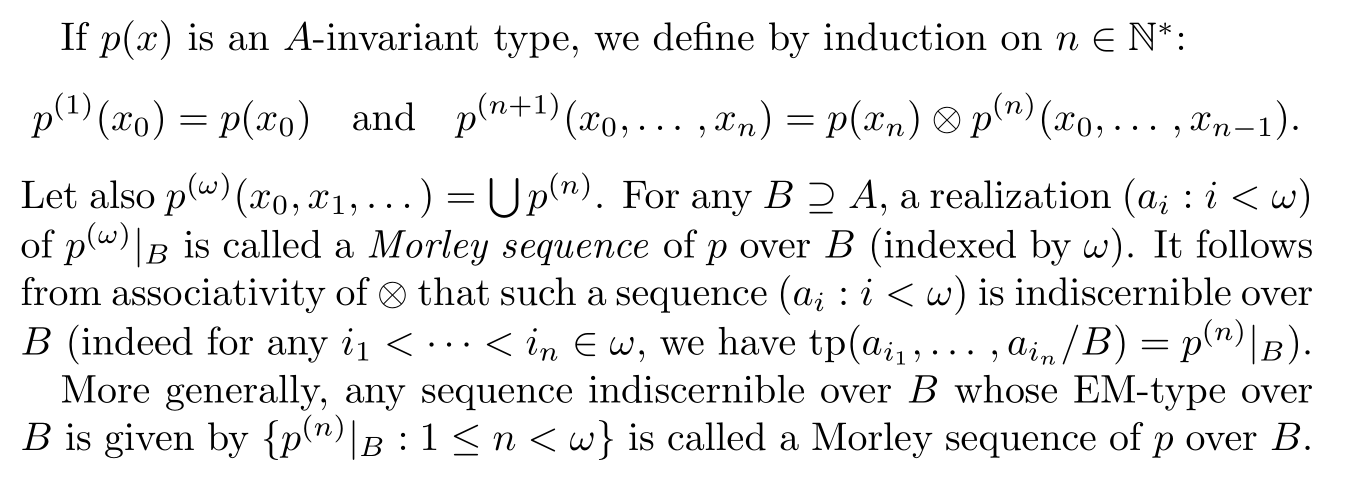}

\subsubsection{Generically stable types} 
	A permutation $\sigma:N\lra N$ acts on  $\SS_\bullet[+N]$ by permuting variables,
	$p(x_1,..,x_n,y_1,..,y_m)\longmapsto  p(x_{\sigma(1)},...,x_{\sigma(n)},y_1,..,y_n)$. 
A type is {\em generically stable} iff $p(x)\tensor p(y)=p(y)\tensor p(x)$, 
	i.e.~iff $\pi_\bullet^p[+1]\circ \pi^p_\bullet$ commutes with the permutation $\sigma:\SS_\bullet(B)\circ[+2]\lra \SS_\bullet(B)\circ[+2]$ permuting the two variables \cite[2.2.2,Theorem 2.29]{Simon}.
	In fact, if $p(x)$ is generically stable, this holds for any lifting $\pi_\bullet^q$ 
	\cite[2.2.2,Proposition 2.33]{Simon}.

\subsubsection{Product of types in a stable theory} 
	Recall that for a definable type $p(-/B)$ \cite[Def.~8.1.4]{TZ} 
	and any $L$-formula $\phi(\bar x,\bar y)$ with parameters in $B$ 
	\cite[Def.~8.1.4]{TZ} defines the formula $d_p\,\bar x \phi(\bar x, \bar b)$ by
\noindent\newline\includegraphics[test]{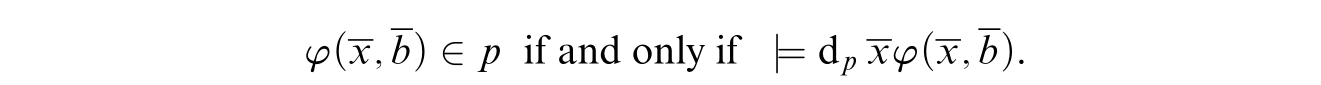}
In \cite{TZ}, the fact that in a stable theory $p(x)\tensor p(y)=p(y)\tensor p(x)$ is 
expressed as
\noindent\newline\includegraphics[test]{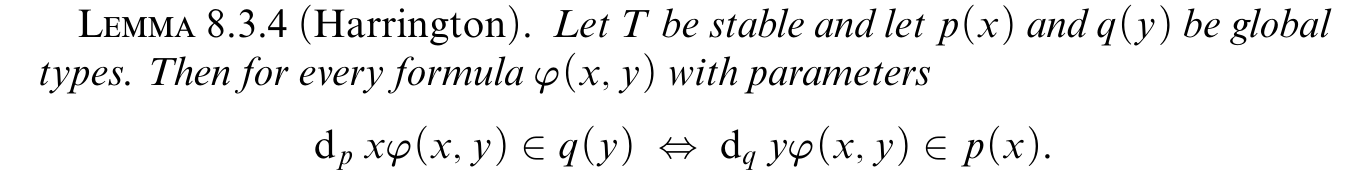}
In simplicial notation, this is represented by the following diagram:
$$
\xymatrix@C=+2.39cm{ &   \SS_\bullet (B) \circ[+1] \ar[r]|{\pi_\bullet^q[+1]} & \SS_\bullet (B) \circ[+2] \ar[d]^{(x,y,z,...)\longmapsto (y,x,z,...)} \\
\SS_\bullet (B) \ar[ru]|{\pi_\bullet^p} \ar[r]|{\pi_\bullet^q} \ar[drr]|{\id}  &   \SS_\bullet (B) \circ[+1] \ar[dr]|{\pr_{2,3,4,...}} \ar[r]|{\pi_\bullet^p[+1]} & \SS_\bullet (B) \circ[+2] \ar[d]^{\pr_{3,4,...}} \\
&&\SS_\bullet (B) }
$$
This corresponds to the following topological picture involving suspension (or cone): 
$$
\xymatrix@C=+2.39cm{ &   \SS_{c.c.}X   \ar@{<-}[r]|{\SS_{c.c.}\pi^q} & \SS_{c.c.}\SS_{c.c.}X   \ar@{<-}[d]^{(x,t_1,t_2)\longmapsto (x,t_2,t_1)} \\
X  \ar@{<-}[ru]|{\pi^p} \ar@{<-}[r]|{\pi^q} \ar@{<-}[drr]|{\id}  &   \SS_{c.c.}X \ar@{<-}[dr]|{} \ar@{<-}[r]|{\SS_{c.c.}\pi^p} & \SS_{c.c.}\SS_{c.c.}X  \ar@{<-}[d]^{} \\
&&X  }
$$

\subsubsection{Definability patterns (speculation)} It is tempting to view definability patterns as some kind of
	analogue of structure (e.g., compact-open topology) 
	on possible liftings $\SS_\bullet(M)\lra \SS_\bullet(M)\circ[+1]$, i.e., by analogy, on the set(space...) of 
	homotopies  $\Hom(\SS X, X)$ which are identity on $X$....
I cannot say more at this stage.

%\subsection{Stability}
%A theory is stable iff each type over a model is definable. This corresponds to the following diagram:

\subsection{Interpretations}
%%%%
%%%%\subsubsection{A definable set as a retract} Let $\phi(-)$ be a 1-ary formula,
%%%%and let $T_\phi$ be the theory of the definable subset $\phi(\CCC)$ with the full induced structure.
%%%%The (usual) Stone space  $\SS_1^{T_\phi}(B)\sqcup \{\star\}$ of $1$-types of $T_\phi$ over $B$ 
%%%%adjoined with an open and closed point, is a retract of the (usual) Stone space 
%%%%$\SS_1^T(B)$ of $1$-types of $T$. 
%%%%
%%%%
%%%%
%%%%and let $r_\phi(\mathfrak C\lra \phi(\mathfrak C)\sqcup \{\star\}$ 
%%%%be the obvious function which is identity on $\phi(\CCC)\subset \CCC$ and 
%%%%takes everything else into a new point $\star$. 
%%%%This function is $\Aut(\CCC/B)$-invariant, and thus defines a map 
%%%%

\subsubsection{Reducts as morphisms} 
A sublanguage $L_0\subset L$ and a subset $B_0\subset B$ of parameters
defines the obvious forgetful 
morphism $\SS_\bullet^{T}(B)\lra \SS_\bullet^{T(L_0)}(B_0)$
remembering only the $L_0(B_0)$-formulas of the types.

\subsubsection{Contractibility and Shelah's expansion by externally definable sets}
This map being contractible (i.e.~fitting into diagram (\ref{SBlift343}))
would mean that each $T(L_0)$-type over $B$ is definable in $L(B)$. 
Roughly, for $B=M$ a model, this means that  $T$ contains the Shelah expansion of $T_0:=T(L_0)$. 
\begin{equation}\begin{gathered}\label{SBlift343}
	\xymatrix{ &  \SS^{T(L_0)}_\bullet(M)\circ[+1]\ar[d]|{\pr_1\times \operatorname{pr}_{2,3,..}} \\
\const_\bullet  \SS^T_1(M)\times 	\SS^T_\bullet(M) \ar[r]|{\operatorname{id}} \ar@{-->}[ru]|{\pi_\bullet} & 
\const_\bullet \SS^{T(L_0)}_1(M) \times \SS^{T(L_0)}_\bullet(M) }\end{gathered}\end{equation}
%If $T$ and $T(L_0)$ are fixed, and 
If $T_0=T(L_0)$ is unstable,
and we take $B_0$ and $B$ to be large enough of the same cardinality,
this has to fail: there are more $T(L_0)$-types over $B_0$
than $T$-formulas over $B$.

\subsubsection{Proving non-intepretability using homotopy theory?} 

Hence, if methods of homotopy theory were able to prove that 
each map between certain simplicial (not contractible) type spaces is contractible,
then we perhaps were able to prove a non-interpretability result in model theory....

\subsubsection{Characterising interpretations simplicially?} 
I have nothing to say. \cite[Theorem 3.1]{Morley} characterises
simplicial type spaces arising from $L_{\omega_1\omega}$-theories
(without explicitly using the words ``functor'' or ``category''). 
See also \cite[\S3]{Levon} and \cite{Kamsma} for a modern exposition of type space functors
in a different context, especially
				\cite[\S3(The type space functor and interpretations of theories)]{Levon} 
				and \cite[Defs.~4.19-20]{Kamsma} which I have not yet read.

\subsection{Shelah's representability}
The meaning of a morphism between two generalised Stone spaces is reminiscent of the notion of
one structure {\em representing} another introduced by
%We also reformulate a corollary of a result of
Shelah  \href{http://mishap.sdf.org/Shelah_et_al-2016-Mathematical_Logic_Quarterly.pdf}{[CoSh:919]}
(we quote  \href{https://arxiv.org/pdf/1412.0421.pdf}{[Sh:1043]})
%on representability of stable theories
%to give a characterisation of stable theories
%in terms of the category we consider.  We note that our reformulation
%corresponding  %more literary
to `try to formalise the intuition %expressed in  \href{http://mishap.sdf.org/Shelah_et_al-2016-Mathematical_Logic_Quarterly.pdf}{[CoSh:919]}
that  ``the class of models of a stable first order theory is not much more complicated than the class
of models $ M=(A, \dots, E_t, \dots)_{s \in I } $                                                                                                             
where $E^M_t$ is an equivalence relation on $A$ refining $E^M_s$ for $s < t$ ; and  $I$ is a linear order of cardinality $\le |T|$''.'
%(we quote \href{https://arxiv.org/pdf/1412.0421.pdf}{[Sh:1043]}).
In \cite[\S3.2.4]{Z1} we  reformulate a corollary of  a characterisation of stable theories in
\href{http://mishap.sdf.org/Shelah_et_al-2016-Mathematical_Logic_Quarterly.pdf}{[CoSh:919]}
and give a more literal formalisation of this intuition: a theory is stable iff there is $\kappa$
such that for each model of the theory there is a surjective morphism to its generalised Stone space
from a structure whose language consists of at most $\kappa$ equivalence relations and unary predicates (and nothing else).
Based on this reformulation we suggest a conjecture with a category-theoretic characterisation of classes of models of stable theories.

It will be interesting to compare this to \href{https://arxiv.org/abs/1810.01513}{[Boney, Erdos-Rado Classes, Thm~6.8]}.

\subsection{Stability as a lifting property}\label{sNOP}
In \cite[\S3.3.2]{Z1} we observe that stability can defined by a lifting property.
Let us very briefly sketch this observation adapted to our current context.
Let $I:=\Bbb Q$ be a countable dense linear order, and let 
$I_\bullet^\leqslant$ and $|I|_\bullet$ be the simplicial sets represented by the linear order, resp.~the set of its elements:
$$ I_\bullet^\leqslant: n\longmapsto \Hom_{preorders}(n^\leqslant, I)$$
$$ |I|_\bullet: n \longmapsto \Hom_{Sets}(n, I)=I^n$$

Recall that a theory is stable iff any infinite indiscernible sequence of $n$-tuples is necessarily 
an indiscernible set, for each $n>0$.
This definition is captured for $n=1$ in $\sSets$ by the following diagram (where we only consider horizontal arrows whose image has unbounded dimension):
$$\xymatrix@C=+5.66cm{  I_\bullet^\leqslant/\!\Aut(I^\leqslant) \ar[d] \ar[r]^{\text{unbounded dimension}} &  \SS_\bullet(\emptyset) \\
|I|_\bullet/\!\Aut(I) \ar@{-->}[ru] 
}$$
One way to turn this diagram into an actual lifting property is to consider it in the category 
$\sPHi$ of simplicial filters with continuous maps defined almost everywhere (see~Definition~\ref{sFdef}), 
equip $I$ with the filter of cofinite subsets, and equip each $I_\bullet^\leqslant(n)$ and 
$|I|_\bullet(n)$ with the coursest filter such that all the simplicial maps 
induced by $1^\leqslant\lra n^\leqslant$ are continuous.

\section{Transcribing Exercise 8.3.3 as simplicial diagram chasing}\label{sec:ExTrans}

We transcribe the Exercise 8.3.3 into the simplicial language. This section is self-contained
and may be read first. 

\subsection{Exercise 8.3.3 in model theoretic language}
We start by quoting in full the Exercise 8.3.3, its solution, 
and the (only) required definition of a definable type.
Fix a theory $T$ and a monster model $\mathfrak C$ of $T$. 
For a subset $B\subset \mathfrak C$, let $S_n(B):=\mathfrak C^n / \Aut_L(\mathfrak C/B)$ denote 
the space of complete $n$-types over $B$. Recall 
%that a subset $U\subset S_n(B)$ is open (and necessarily also closed) 
%iff there is an $L(B)$-formula $\phi(\bar x)$
	the topology on $\SS_n(B)$ is generated by open sets
	$U_\phi=\{p(\bar x)\in S_n(B): \phi(\bar x)\in p(\bar x)\}$
	where $\phi$ varies though arbitrary formulas with parameters in $B$.

\includegraphics[test]{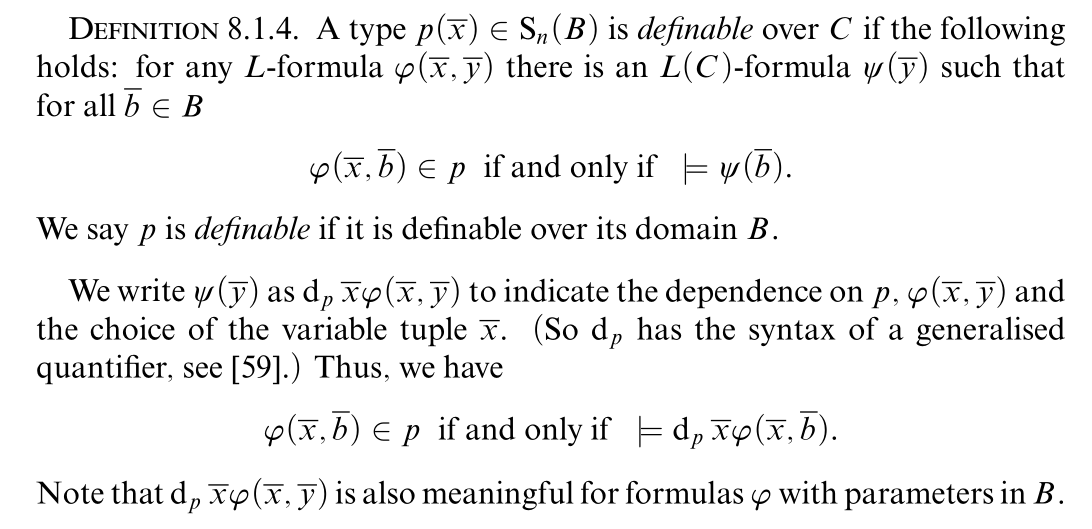}
%\pagebreak
\includegraphics[test]{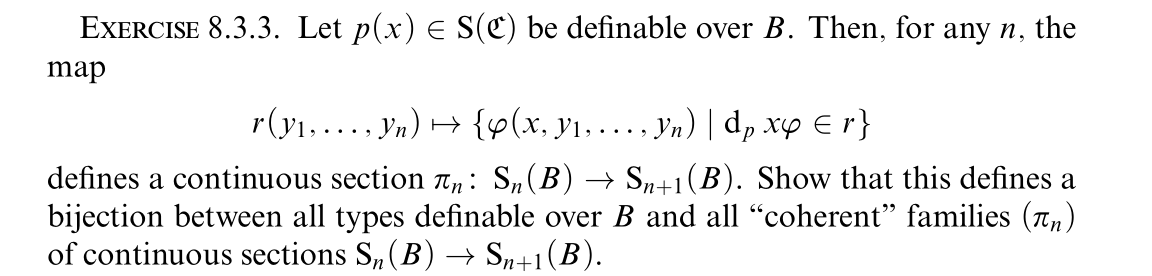}
%\pagebreak
\includegraphics[test]{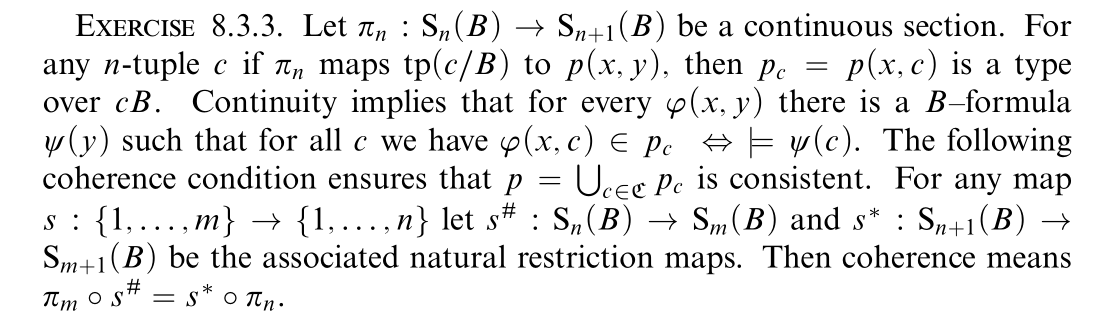}
%\newpage

\subsection{Exercise 8.3.3 in simplicial language}\label{ExTrans} 
Let $\Delta$ denote the category of non-empty finite linear orders 
denoted by $\{1<..<n\}$. Let $[+1]:\Delta\lra \Delta$ be the 
{\em decalage} endomorphism
adding a new least element to each finite liner order:
$$[+1]: \{1<..<n\} \longmapsto \{0<1<...<n\}$$
$$f:\{1<...<m\}\lra \{1<...<n\} \longmapsto f[+1](0):=0, f[+1](l)=:l$$ 
We denote the finite linear order $\{1<...<n\}$ either by $n^\leqslant$,
or by $[n-1]$, as is standard in simplicial literature. 

Let $\SS_\bullet(B):\Delta^{op}\lra \TTop$ be the functor which sends 
each finite linear order $\{1<...<n\}$ to $S_n(B)=\mathfrak C/\Aut_L(\mathfrak C/B)$. 
Recall that functors $\Delta^{op}\lra \TTop$ are called {\em simplicial objects of the category $\TTop$},
or sometimes {\em simplicial topological spaces}.
In fact, the functor  $\SS_\bullet(B):\Delta^{op}\lra \TTop$  factors 
though the embedding of the category $\Delta^{op}$ into the opposite 
of the category $\fSets_{\neq\emptyset}$ of non-empty finite sets,
and let $\tilde\SS_\bullet(B):\fSets_{\neq\emptyset}\lra \TTop$ 
be the corresponding functor.

A ``coherent'' family of sections $\pi_n:S_n(B)\lra S_{n+1}(B)$ 
defines a natural transformation $\tilde\SS_\bullet(B) \lra \tilde\SS_\bullet(B)\circ [+1]$. 
Indeed, the coherence condition $\pi_m\circ s^\#=s^*\circ \pi_n$ 
is precisely the defining property of a natural transformation
$\tilde\SS_\bullet(B) \lra \tilde\SS_\bullet(B)\circ [+1]$. 
%In particular, 
In fact, because of symmetry it is equivalent to require the coherence conditions
only for non-decreasing maps $s:\{1,...,m\}\lra \{1,...,n\}$:
for any permutation $\sigma:\{1,...,m\}\lra\{1,...,m\}$, the equality 
$\pi_m\circ s^\#(q(y_1,...,y_n)=s^*\circ \pi_n(q(y_1,...,y_n)$
 is equivalent to 
$\pi_m\circ s^\#(q(y_{\sigma(1)},...,y_{\sigma(n)})=s^*\circ \pi_n(q(y_{\sigma(1)},...,y_{\sigma(n)})$. 
Hence, it is equivalent to say that a ``coherent'' family of sections $\pi_n:S_n(B)\lra S_{n+1}(B)$ 
defines a natural transformation $\SS_\bullet(B) \lra \SS_\bullet(B)\circ [+1]$. 

These coherence conditions is equivalent to the consistency of the global type 
$$p(x/\mathfrak C) =\bigcup\limits_{c\in\mathfrak C^n,n>0}p_{c}=\{\varphi(x,\bar c): \phi(x,\bar y)\in \pi_n(\tp(\bar c/B))\}$$
Indeed, consider a finite collection $\varphi_i(x,\bar c_i)\in p(x/\mathfrak C)$, $i<n$ of formulas.
Consider the type $\pi_N\tp(\bar c_1,...,\bar c_n/B)$ of the joint $N$-tuple $(\bar c_1,...,\bar c_n)$. 
For an appropriate map $s:\operatorname{length}(\bar c_i)^\leqslant \lra N^\leqslant$,  
the coherence conditions (=funtoriality) implies that 
$\pi_{\operatorname{length}(\bar c_i)}(\bar c_i)=\pi(s)(\pi_N(\bar c_1,...,\bar c_n))$,
hence $\varphi_i(x,\bar c_i)\in \pi_N(\bar c_1,...,\bar c_n)$, which is consistent.

By construction, the global type $p(x/\mathfrak C)$ is $B$-invariant.

Hence, there is a bijection between global $B$-invariant types
and ``coherent'' families of (possibly discontinuous) sections $(\pi_n)$,
or, equivalently, (possibly discontinuous)  sections $ \SS_\bullet(B)\lra \SS_\bullet(B)\circ[+1]$
of simplicial sets. 
%by \cite[2.2,Example 2.17]{Simon}, a global type is 

Therefore, we can reformulate Exercise 8.3.3 as follows:

\begin{quote}
{\sc Exercise 8.3.3}. Let $p(x)\in S(\mathfrak C)$ 
	be definable over $B$. %// where C is the monster model
Then, for any $n$, the map
$$r (y_1, . . . , y_n )\longmapsto \{\phi(x, y_1, . . . , y_n ) | d_p\, x\phi \in r\}$$
or, equivalently, 
$$r \longmapsto p\tensor r$$
defines a continuous section  $\pi_n : S_n (B)\lra S_{n+1} (B)$. 
Show that this defines a bijection between 
	\begin{itemize}\item all types definable over $B$ 
			\item all ``coherent'' families $(\pi_n)$
of continuous sections $S_n (B)\lra S_{n+1} (B)$.
\item 
lifting arrows in the diagram of simplicial topological spaces
			\begin{equation}\begin{gathered}\label{SBlift}\xymatrix{ & \SS_\bullet(B)\circ[+1]\ar[d]|{\operatorname{pr}_{2,3,..}} \\
			\SS_\bullet(B) \ar[r]|{\operatorname{id}} \ar@{-->}[ru]|{\pi_\bullet} & \SS_\bullet(B) }
			%\label{SBliftCAut}
			\xymatrix{ &  \mathfrak C_\bullet/\!\!\Aut(\mathfrak C/B) \circ[+1]\ar[d]|{\operatorname{pr}_{2,3,..}} \\
			\mathfrak C_\bullet/\!\!\Aut(\mathfrak C/B)  \ar[r]|{\operatorname{id}} \ar@{-->}[ru]|{\pi_\bullet} 
			& \mathfrak C_\bullet/\!\!\Aut(\mathfrak C/B) }\end{gathered}
	\end{equation}
	\end{itemize}
Moreover,  
show that this defines a bijection between 
	\begin{itemize}\item all global types invariant over $B$ 
			\item all ``coherent'' families $(\pi_n)$
of possibly discontinuous sections $S_n (B)\lra S_{n+1} (B)$.
\item 
lifting arrows in the diagram of simplicial sets
			\begin{equation}\begin{gathered}\label{SBlift2}\xymatrix{ & \SS_\bullet(B)\circ[+1]\ar[d]|{\operatorname{pr}_{2,3,..}} \\
			\SS_\bullet(B) \ar[r]|{\operatorname{id}} \ar@{-->}[ru]|{\pi_\bullet} & \SS_\bullet(B) }\end{gathered}\end{equation}
	\end{itemize}
	\end{quote}
Note that by \cite[2.2,Example 2.17]{Simon} any finitely consistent global type over $B$ is 
necessarily $B$-invariant. As every type over a model is necessarily finitely consistent,
it implies that each global type over a model gives rise to a diagram (\ref{SBlift2}) in $\sSets$.
Hence, a notion of continuity is essential to be able to define definability. 

\section{Background for a homotopy theoretic interpretation of Exercise 8.3.3}

It is well-known that the simplicial formula (\ref{SBlift}) defines
homotopy triviality in $\sSets$, as was pointed to us by V.Sosnilo. 
Note that the diagram (\ref{SBlift})  can be written in the category 
of simplicial objects of an arbitrary category.  

Unfortunately, we understand very little about this formula,
and request comments on it from our homotopy theory readers.

Below in \S\ref{TopSimpHomotopy} we explain that the formulas (\ref{SBlift}) and (\ref{SBlift33}) defines the usual notion
of a map being contractible in the category of topological spaces
when applied to the singular complexes of sufficiently nice topological spaces.

In \S\ref{ConvergenceSimpHomotopy} we say that the same formula 
defines the notion of convergence in the category of simplicial objects
of a category of filters.

\subsection{Simplicial homotopy in the category of topological spaces}\label{TopSimpHomotopy} 
Indeed, in the category $\TTop$ of topological spaces, 
	a homotopy
	contracting a space $F$ in a space $X$ 
	to  a point (i.e.~a map $h:F\times [0,1]/F\times \{1\} \lra X$
	from the cone of $F$ to $X$), gives rise
	to a map $$h_\bullet: \sing_\bullet F \lra \sing_\bullet X[+1]$$
	of singular complexes lifting the map $(h_{|F\times\{0\}}) _\bullet: \sing_\bullet F \lra \sing_\bullet X$, defined as follows.
	This map  
		takes each 
	 $\delta: \Delta^n \to F$ in $\sing_\bullet F((n+1)^\leqslant)$ to
		$h_*(\delta):\Delta^n\times [0,1]/\Delta^n\times\{1\}\to X$ in $\sing_\bullet X((n+2)^\leqslant)$ defined by
		$h_*(\delta)( x,t ):= h(\delta(x),t)$.
	\begin{equation}\begin{gathered}\label{SingLift}
		\xymatrix
		{
			 {} && {\sing_\bullet X\circ[+1]}\ar[dd]|{pr_{2,3,...}}\\ \\
			 {\sing_\bullet F}\ar@{-->}[uurr] |{h _\bullet}  \ar[rr]|{} && {\sing_\bullet X}
		 }
	\end{gathered}\end{equation}

By a nice topological space we mean any class of spaces such a map $h:F\lra X$ is contractible (=null-homotopic) iff it is weakly contractible; Whitehead's theorem says that CW-spaces are nice in this sense. 

	\begin{propo} If $F$ and $X$ are nice, %enough, 
		and $F$ is connected, 
	than a map $h_0:F\lra X$ factors though 
the cone of $F$ as 
	$$F \lra F\times  [0,1]/F \times\{1\}\lra X
$$
iff the induced map $\sing_\bullet X \lra \sing_\bullet X$ of singular complexes 
	factors through the decalage $\pr_{2,3,..}:\sing_\bullet X\circ[+1]\lra \sing X _\bullet$.
	\end{propo}
	\begin{proof}
%consider an example when both
%$F_\bullet$ and $X_\bullet$ are singular complexes of topological spaces:
Recall  that the singular complex is defined using simplices
$$\Delta^n=\Hom_{\text{preorders}}([0,1]^\leqslant, (n+1)^\leqslant)$$ as ``test spaces'':
\begin{center}
	$\sing_\bullet F((n+1)^\leqslant):=\Hom_\text{Top}( \Delta^n, F)$,\\
	$\sing_\bullet X((n+1)^\leqslant):=\Hom_\text{Top}( \Delta^n, X)$, \\
	$\sing_\bullet X\circ[+1]((n+1)^\leqslant)=\Hom_\text{Top}( \Delta^n\times [0,1]/{\Delta^n\times\{1\}}, X)$\\
\end{center}
where $n\geqslant 0$ and  $ \Delta^n\times [0,1]/{\Delta^n\times\{1\}}$ is
the cone of $n$-simplex $\Delta^n$.%=\Hom_{\text{preorders}}([0,1]^\leqslant, (n+1)^\leqslant)$.

		To define  a lifting $h_\bullet$ 
		given a map $h:F\times [0,1]/F\times \{1\} \lra X$, 
		take each $\delta: \Delta^n \to F$ in $F_\bullet((n+1)^\leqslant)$ to
		$h_*(\delta):\Delta^n\times [0,1]/\Delta^n\times\{1\}\to X$ in $X_\bullet((n+2)^\leqslant)$ defined by
		$h_*(\delta)( x,t ):= h(\delta(x),t)$.

		To see the other direction, note that 
		a map 
		$h_\bullet:F_\bullet\lra X_\bullet[+1]$ takes a singular simplex
		$\delta:\Delta^n\lra F$ of $F$ 
		into 
		a singular simplex $h_\bullet(\delta):\Delta^{n+1}=\Delta^n\times [0,1]/\Delta^n\times\{1\}\lra X$ of $X$ 
		such that
		%$\delta\circ h_0=h_\bullet(\delta)_{|\Delta^n\times \{0\}}$, i.e.~each
		$\delta=\pr_{2\!<\!3\!<..\!<\!n} h_\bullet(\delta)$, i.e.
                $\delta=h_\bullet(\delta)_{|\Delta^n\times \{0\}}$, 
		and thereby
		each $\delta:\Delta^n\lra F\lra X$ factors through the cone of $\Delta^n$.
		A verification using functoriality shows that the same factorisation holds for
		a map $\delta':\mathbb S^n = \partial \Delta^{n+1}$ from
		any connected sphere $\mathbb S^n = \partial \Delta^{n+1}$, $n>0$, 
		which means exactly that $h_0$ is weakly contractible,
		and for nice topological spaces  contractible and weakly contractible are equivalent.
\end{proof}

Let $\const_\bullet F$  denote the constant functor $\Delta^{op}\lra \TTop$,
$\const_\bullet F(n^\leqslant):=F$ for all $n>0$.

	\begin{propo} Assume topological spaces $F$ and $X$ are nice. % enough.

	A map $h_0:F\lra X$ is homotopic to a constant map, i.e.~factors though 
the cone of $F$ as 
	$$F \lra F\times  [0,1]/F \times\{1\}\lra X
$$
iff the induced map $\const_\bullet F \times \sing_\bullet F \lra \const_\bullet X \times \sing_\bullet X$ of singular complexes 
	factors through the decalage $\pr_{2,3,..}:\sing_\bullet X\circ[+1]\lra \const_\bullet X \times 
\sing X _\bullet$.
	\begin{center}$
		\xymatrix
		{
			 {} && {\sing_\bullet X\circ[+1]}\ar[dd]|{\pr_0 \times  \pr_{2,3,...}}\\ \\
			 {\const_\bullet F \times \sing_\bullet F}\ar@{-->}[uurr] |{h _\bullet}  
			 \ar[rr]|{} && 
			 {\const_\bullet X \times \sing_\bullet X}
		 }
	$\end{center}

	\end{propo}

In particular, a nice (possibly disconnected) space $X$ is contractible iff 
			$$
			\xymatrix@C=+2.39cm{ &   \Sing_\bullet X\circ[+1] \ar[d]^{\pr_{2,3,...}} \\
\const_\bullet X \times \Sing_\bullet X  \ar@{-->}[ru]|{} \ar[r]|{\id} & \const_\bullet X \times \Sing_\bullet X } 
$$

\subsection{Convergence as being contractible}\label{ConvergenceSimpHomotopy} 
In \cite{L} we observe we associate with a sequence $(a_i)_i$ of points of 
a topological %or metric 
space $X$ a morphism in the category 
$\sPhi$ of simplicial objects in the category $\PPhi$ of filters (cf.~Definition~\ref{sFdef})  
such that it factors as in formula (\ref{SingLift}) iff the sequence is convergent; 
moreover, limits of the sequence correspond precisely to the liftings. 
In fact, we first wrote \cite[\S3.2]{mintsGE} the simplicial diagram (\ref{SingLift})
when transcribing the definition of a limit of a filter on a topological space in \cite{Bourbaki},
but were sadly unaware of its connection to contractibility before a remark by V.Sosnilo.

%%%%\section{A homotopy theoretic interpretation of Exercise 8.3.3}
%%%%
%%%%The discussion above suggests that, in a certain precise simplicial sense, 
%%%%\begin{quote} Exercise 8.3.3 says that {\em a global type definable over a model $M$ is the same 
%%%%	as a homotopy contracting each connected component of the simplicial space $\SS_\bullet(M)$ of types over the model $M$}.
%%%%\end{quote} 
%%%%
%%%%Recall that a theory is stable iff each type over a model $M$ is definable. 
%%%%This is expressed by the following diagram: 
%%%%\begin{equation}\begin{gathered}\label{SBlift3}
%%%%	\xymatrix{ &  \SS_\bullet(M)\circ[+1]\ar[d]|{\pr_1\times \operatorname{pr}_{2,3,..}} \\
%%%%\const_\bullet  \SS_1(M)\times 	\SS_\bullet(M) \ar[r]|{\operatorname{id}} \ar@{-->}[ru]|{\pi_\bullet} & 
%%%%\const_\bullet \SS_1(M) \times \SS_\bullet(M) }\end{gathered}\end{equation}
%%%%Here $\const_\bullet \SS_1(B)$ denotes the constant functor $\Delta^{op}\lra \TTop$, 
%%%%$\const_\bullet \SS_1(B)(n^\leqslant):=\SS_1(B)$ for all $n>0$. 
%%%%Indeed, the diagram says that each 1-type over $M$ gives rise to a ``coherent'' family of 
%%%%continuous sections, and is therefore definable. Informally, one might perhaps think 
%%%%that this says that the Stone space can be contracted to each of its points.
%%%%
%%%%Hence, we may say that Exercise 8.3.3 implies that
%%%%\begin{itemize}
%%%%	\item {\em a theory is stable iff its simplicial space of types is contractible, \\
%%%%	in a certain precise simplicial sense}
%%%%\end{itemize} 
%%%%
%%%%

\section{Research directions. Test problems.}\label{TestProblems}

We suggest a couple of test problems which might be used to guide development of the simplicial reformulations of model theory.

\subsection{State and prove simplicially the Fundamental Theorem of Stability Theory} 
The Fundamental Theorem of Stability Theory claims equivalence of 
two definitions of stability theory admitting simplicial reformulations:
(*) each type over a set is definable, and (**) no formula has the order property. 
The first definition is what this note is about, and the second defitition 
is quite close to the reformulations in terms of the lifting property
discussed in \cite{Z1,Z2}, esp.~\cite[\S3.3.2]{Z1},\cite[\S17]{Z2} see also \S\ref{sNOP}.
Moreover, it follows from a general theorem about compactness,
namely the Grothendieck's double limit theorem \cite[Thm.6]{Groth}, 
as explained in \cite{BenYaacov,Starchenko}. 
Note that there is a definition of compactness involving diagram similar to that 
used to define definability \cite{L}.

Can one give a purely simplicial or homotopy theoretic proof of this theorem ? 

One immediate difficulty is that it is not quite clear what category 
one should work in: the reformulations in \cite{Z1,Z2} use the category $\sPhi$ of simplicial 
filters and continuous maps (though \S\ref{sNOP} suggests it might be better to use 
almost everywhere continuous maps),
whereas here we use the category $\sTop$ of simplicial topological spaces or its full subcategory $\spfSets$
of simplicial profinite Hausdorff compact spaces. See also speculations in \S\ref{forking-continuity}(Forking as a notion of continuity?).

\subsection{Kim-Pillay characterisation of non-forking in terms of independence relation}
Rewrite simplicially the characterisation of simple theories \cite[\S4, Def.~4.1]{Kim-Pillay} 
and stable theories \cite[Axioms~0-4]{Harnik-Harrington}
in terms of parameter sets and types as morphisms to $\SS^T_\bullet(\emptyset)$.
Doing so appears straightforward, and the difficulty lies in identifying the category-theoretic notion 
that such a translation would lead to. Note that \cite[Def.~2.1]{Lieberman-Rosicky-Vasey} and \cite{Kamsma2} use a different approach 
to reformulate these notions category-theoretically.

\subsubsection{Forking as a notion of continuity?}\label{forking-continuity}
\cite{Harnik-Harrington} observed stability of a first-order theory $T$ can be characterised  in terms of 
a class of distinguished extensions $p\subset q$ of its complete types (here $p\subset q$ means that each formula in $p$ is also in $q$)
over arbitrary sets of parameters, see also \cite[Theorem 8.5.10(Characterisation of Forking)]{TZ}.
They denote this relation
by $p \sqsubset q$ between the complete types of a theory $T$, and the intuition is that $p\sqsubset q$ means that 
$q$ is a {\em free} extension of $p$ to a larger set of parameters. 
They prove that a theory $T$ is stable iff there is a relation $p\sqsubset q$ on its complete types over arbitrary parameters
satisfying the following axioms:
\newline\noindent\includegraphics[test]{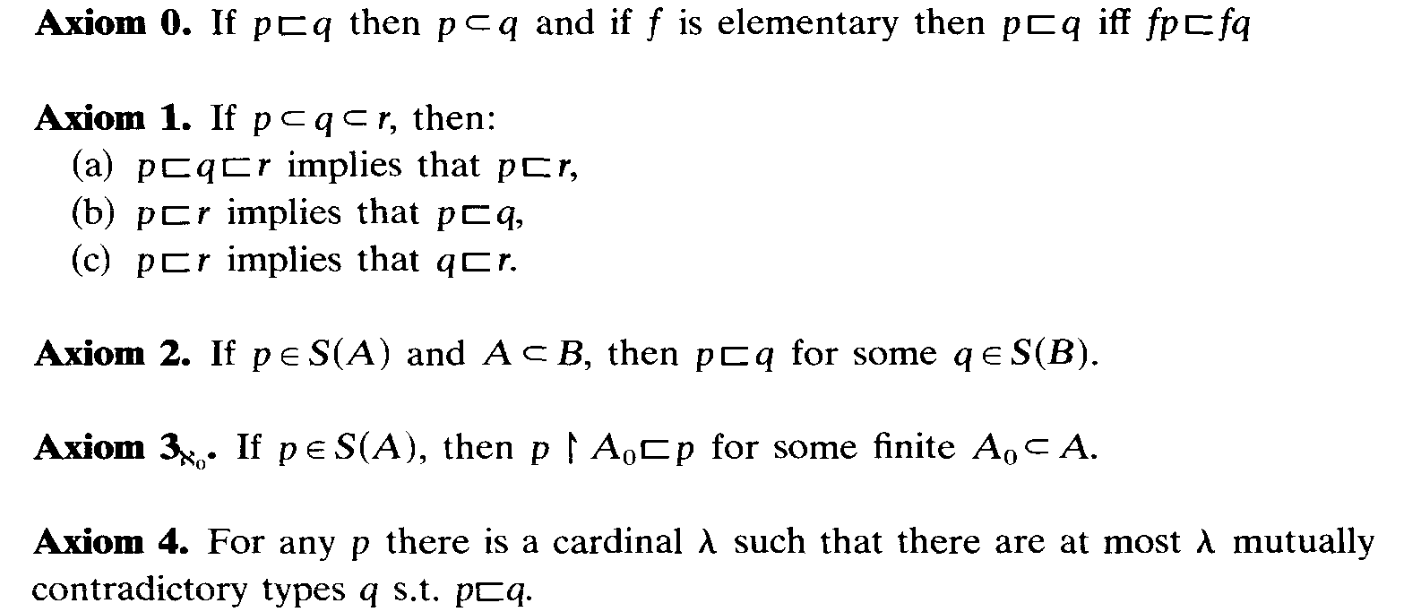}
\newline\noindent\includegraphics[test]{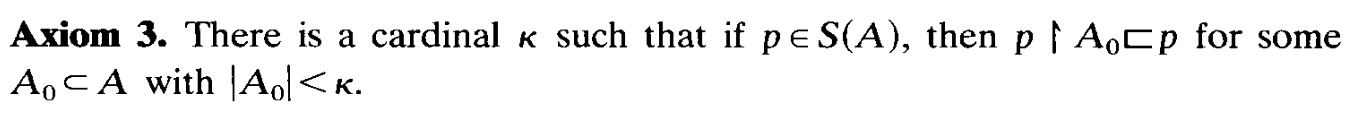}
In simplicial terms, a relation $p\sqsubset q$ is a class of distinguished diagrams in $\sSets$
$$\xymatrix@C=2.239cm{ |A|_\bullet \ar[r]|(0.3)p \ar[d]|{A\subset B} & \SS_\bullet(\emptyset)\circ[+N]\ar[d]|{\pr_{N+1,N+2,...}} \\
|B|_\bullet \ar[r]|{\tp(B/\emptyset)} \ar[ru]|q & \SS_\bullet(\emptyset)}
\xymatrix{ & } 
\xymatrix@C=2.239cm{ |A|_\bullet \ar[r]|(0.3)p \ar[d]|{A\subset B} & \SS_\bullet(B)\circ[+N] \ar[d]|{\pr_{N+1,N+2,...}}  \\
|B|_\bullet \ar[r]|{\tp(B/B)} \ar[ru]|q & \SS_\bullet(B) 
}$$
A notion of continuity defines an extra structure on $\sSets$ which does provide a class of distinguished diagrams, 
and one wonders if this is a useful point view on non-forking. 
%One wonders if such a choice of distinguished diagrams is related to 
A standard way to define something like a topology on a category is provided by the notion of a Grothendieck topology on a category, 
i.e.~a choice of distinguished families $\{f_i:U_i\lra U\}_i$ of morhpisms called {\em coverings}. 
Note that Axiom~3 reminds of accessible categories.

\subsubsection{Stability in terms of an independence relation.}
We have nothing to say but only quote some definitions with a hope that the reader may recognize the diagrams involved.
\cite[Theorem 5.2, cf.~also Theorem 3.2 and Def.~4.1]{Kim-Pillay} characterise the class of simple theories. 
We shall quote the axioms of an independence relations. \cite[Theorem 3.2(Independence Theorem over a model)]{Kim-Pillay} 
states a non-triviality condition of the independence relation characterising simple theories. 
Setting $tp(a/A) \sqsubset tp(a/BA)$ iff $(a,B,A) $ is independent relates the independence relation and the non-forking extensions above.
\newline\noindent\includegraphics[test]{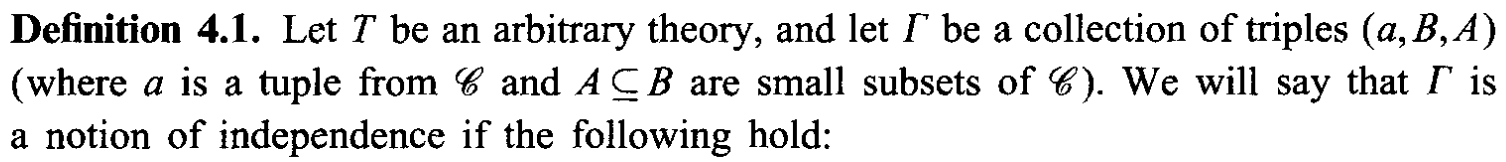}
\newline\noindent\includegraphics[test]{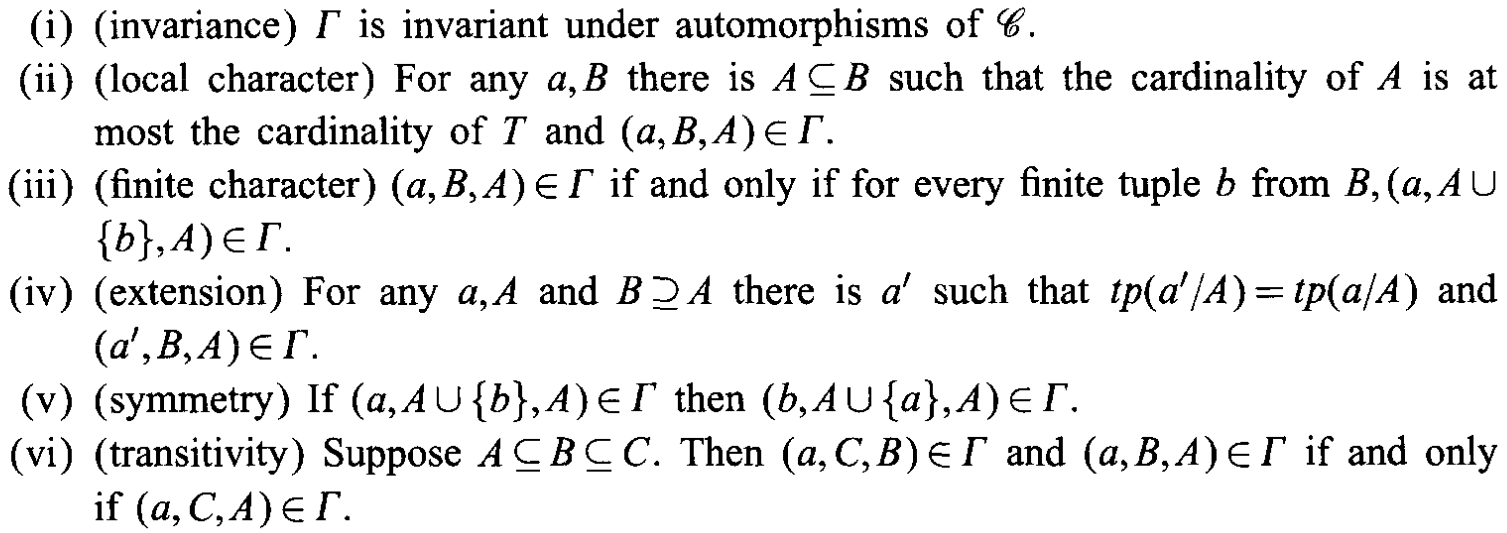}
In simplicial terms, this describes a collection of distinguished diagrams of form
$$\xymatrix{ |a| \vee |B|_\bullet \ar@{<-}[d] \ar@{<-}[r] \ar[dr]  & |a|_\bullet \ar[d] \\
|B|_\bullet \ar[r] & \SS_\bullet(A) }$$
%We have nothing to say here but hope that maybe some readers will recognize this diagram.

\subsection{Reformulate superstability, NIP, categoricity, simplicity, properties of an independence relation and non-forking, Kim-dividing and Kim-forking, excellence...} 
Reformulate some of the classical theory. In particular, reformulate simplicially model theoretic properties of structures related to combinatorics or algebraic geometry, such as those related to 
Elkes-Szabo or peudoexponentiation.

%%%%
%%%%
%%%%In the paper [1] Ita  Ben Yaacov observed that the Fundamental The-
%%%%orem of Stability Theory (FTST for short) follows from Grothendiecks
%%%%double limit theorem (see [2, Theorem 6]), and in fact Grothendiecks
%%%%Theorem implies a strong version of FTST where a formula  is as-
%%%%sumed to be stable in a given structure M and not in the whole theory.

\section{Appendix. Requests for comments}\label{Questions}

In the appendix I take the liberty to present questions which I would ask in 
a private conversation or email.

\subsection{Requests for comments from a homotopy theorist} 

I do not know much about the simplicial formula (\ref{SBlift}) used to define 
contractibility, convergence, and stability of a theory. 
Essentially, any reference to a general theory would be welcome. 

Does homotopy theory suggest a point of view or technique for dealing with stable theories ? 
Particularly in view of the connection between the Borel construction 
of a group action and the simplicial Stone space
mentioned in Remark~\ref{BGXasSB}. For example, did anyone consider
the Borel construction $\BBB_\bullet(\operatorname{Gal}(\bar{\Bbb Q}/\Bbb Q),\Bbb Q)$
of the Galois action on $\bar{\Bbb Q}$ ?
Probably $\BBB_\bullet(\operatorname{GL}(V), V)$ for a vector space $V$, 
is standard to consider, but how would it relate to the Stone space $\SS^{\operatorname{VectorSpaces}}_\bullet$ of the theory of vector spaces...
%Note that the construction of the simplicial type space $\SS_\bullet(B)(n^\leqslant):=\mathfrak C^n/\Aut(\mathfrak C/B)$ is somewhat similar to the construction of $B(G,X)_\bullet(n^\leqslant):=X\times G^n/G$ 
%of a group representation (group $G$ acts on $X$).

I remark that stability (and a number of other properties of theories) 
can be defined by a lifting property with respect to an explicitly given morphism,
in a somewhat similar category $\sPhi$ of simplicial objects in the category $\PPhi$ of filters
\cite{Z1,Z2}.

I should explicitly say that I am no expert in homotopy theory, and solicit collaboration.

\begin{question}[Background on our simplicial formula for contractibility] \begin{itemize} 
	\item Find a good reference discussing this simplicial formula and decalage..
        \item This simplicial formula can be interpreted in the category of simplicial 
		objects of an arbitrary category, and thus defines a notion of a map being contractible.
		Did anyone study this formula as a definition of contractibility ? Is it well-behaved ?
		We do know that in $\sSets$ it does define a standard notion of contractibility 
		for fibrant simplicial sets.
	\item  Is there a similar formula using endomorphisms of $\Delta$ 
		which defines when two maps are homotopy equivalent
		in the category of simplicial objects of an arbitrary category ? 
\end{itemize}	
\end{question}

\subsection{Requests for comments from a model theorist}

Where to take this further ? What notions in model theory look as if they might be added to our little glossary ?  

An obvious wish is to apply methods or intuitions of homotopy theory in model theory. 
Say, make a homotopy theory calculation in model theory.

\subsubsection{Morley sequences as spectra in stable homotopy theory ?} 
Both notions involve a sequence and taking a suspension at each step. Is there any analogy ?

\subsubsection{Does the classifying space  $$\BBB_\bullet(\CCC,\AutCB)$$ appear in model theory?} 
Does Remark~\ref{BGXasSB} provoke any associations in model theory ? 

\subsubsection{A technical question: the same diagram with different notion of continuity} 
\cite{Z1,Z2} show that stability and a number of other notions are defined by lifting properties
not in in the category $\sTop$ of simplicial topological spaces, but in the category $\sPhi$
of simplicial objects in the category of filters. Interpreting (\ref{SBlift}) in that category 
leads to a different property of type which I state below. Is it familiar ? 

\begin{question} Is the following property of types familiar ? 
It seems very much as an analogue of non-dividing but only for elements, not tuples, 
\cite[Cor.7.1.5]{TZ}. Note it is expressed as a lifting property in $\sPhi$ in 
	\cite[5.3.2]{Z1}. 

\begin{itemize}
	\item (oversimplified)
 A global type $p(x/\mathfrak C)$ invariant over $B$ such that
	if $\bar c$ is indiscernible over $B$, then 
	$p(x/\mathfrak C)$ contains all the formulas saying 
	that $\bar c$ is indiscernible over $x$. 
\item  
 A global type $p(x/\mathfrak C)$ invariant over $B$ such that
		\begin{itemize}\item 
for each length $l>0$, for each formula $\varphi(x,\bar y,\bar b)$, 
				$\bar b\subset B$, %\in p(x/\mathfrak C)$ 
there are finitely many formulas $\psi_i(-,\bar b_i)$, $\bar b_i\subset  B$, $0<i<n$, 
such that
				\begin{itemize}
					\item	for any tuple $\bar c\subset \mathfrak C$ of length $l$, % same as $\bar c$, 
	if $\bar c$ is indiscernible wrt each $\psi(-,\bar b_i)$, $0<i<n$, 
	then $p(x/\mathfrak C)$ contains the formula saying that $\bar c$ 
						is indiscernible wrt $\phi(x,-,\bar b)$.
				\end{itemize}\end{itemize}

\item same as above, but instead of indiscernibility wrt finitely many formulas require
	extending to an arbitrary long finite tuple indiscernible wrt finitely many formulas. 
				Thus, it now reads:

 A global type $p(x/\mathfrak C)$ invariant over $B$ such that
		\begin{itemize}\item 
for each lengths $l<l_1$, for each %finitely many formulas $\varphi_i(x,\bar y,\bar b_i)$, $0<i<n$, 
%				$\bar b\subset B$, %\in p(x/\mathfrak C)$ 
 finite set $\Theta$ of formulas over $B$
				there are $l_2>0$ and %finitely many formulas $\psi_j(-,\bar b_j)$, $\bar b_j\subset  B$, $0<j<N$, 
 finite set $\Delta$ of formulas over $B$
				such that
				\begin{itemize}
					\item	for any tuple $\bar c\subset \mathfrak C$ of length $l$, % same as $\bar c$, 
	if $\bar c$ extends to some finite  tuple of length $l_2$  indiscernible wrt $\Delta$, %each $\psi(-,\bar b_j)$, $0<j<N$, 
	then $p(x/\mathfrak C)$ contains the formula saying that $\bar c$ extends to a tuple $\bar c \bar c'$ of length $l_1$
						 indiscernible wrt $\Theta$ %each $\phi_i(x,-,\bar b_i)$, $0<i<n$.
				\end{itemize}\end{itemize}\end{itemize} 
				\end{question}

\subsubsection{References to simplicial type spaces in model theory?}
The only three references to simplicial type spaces I know,
				are by (Michael Morley. Applications of topology to $L_{\omega_1\omega}$. 1974) \cite{Morley},
and 
				by (Levon Haykazyan.
{\em Spaces of Types in Positive Model Theory}. J. symb. log. 84 (2019) 833-848.) \cite{Levon},
				and (Mark Kamsma. 
				{\em Type space functors and interpretations in positive logic}. 2020).
				The latter two \cite{Levon,Kamsma}
				mention simplicial type spaces under the name of {\em type space functors}
				and consider them in the context of positive logic.
				We particularly draw attention to 
	\cite[\S3(The type space functor and interpretations of theories)]{Levon} 
				and \cite[Defs.~4.19-20]{Kamsma} which I have not yet read.

				\cite{Morley} calls them {\em type structures} associated to an $L_{\omega_1\omega}$-theory, 
				but never uses words ``functor'' or ``category'' explicitly.
				Consider the following wording by Morley used to introduce notions necessary to charactercise simplicial spaces (called ``type structures'') associated with an $L_{\omega_1\omega}$-theories. 
				%Note that $\Pi$ and $\lambda$ refere essentially to the decalage endomorphism,
				%resp.~the decalage endomorphism adding a maximal rather than minimal element.
\noindent\newline\includegraphics[test]{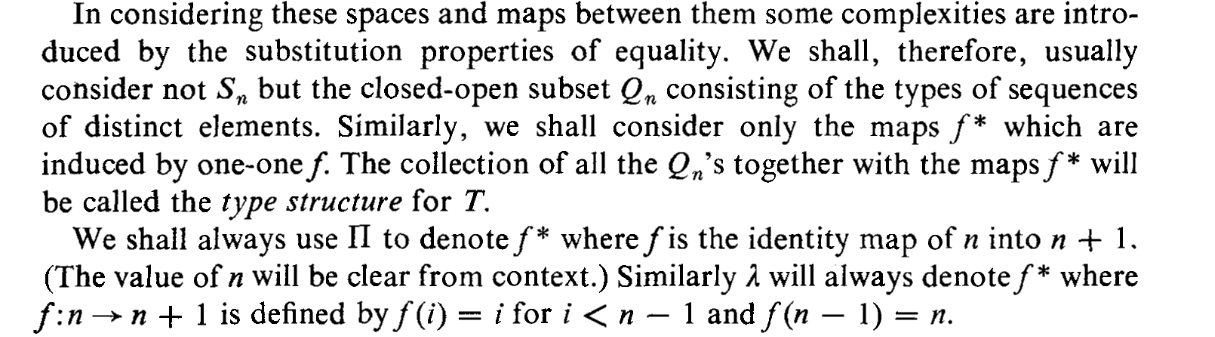}

Is there anything else ? In particular, about interpretations as maps of type spaces.

				\subsubsection{A concise definition of simplicial Stone spaces of types in the category of filters\label{sFdef}\label{A}}  
Let me now define the category $\PPhi$ of filters with continuous maps, 
and the category $\PHi$ of filters with continuous maps defined almost everywhere. 
				\begin{defi}\label{A}
An object of $\PPhi$ is a set equipped with a filter. 
A morphism $f:(X,\mathcal F) \lra (Y,\mathcal G)$ is a map 
$f:X\lra Y$ of the underlying sets such that the preimage of a big set is big, 
i.e.~$\{ f\inv(U): U\in \mathcal G\}\subset \mathcal F$. 
We call such maps of filters {\em continuous}, as it enables us to say 
that a map of topological spaces is continuous iff the induced maps of neighbourhoods filters
are continuous. 

Let $\PHi$ denote a category of filters where morphisms are defined only on big subsets, 
where we identify maps which coincide on a big subset. That is, 
$\PHi$ and $\PPhi$ have the same objects, and in $\PHi$
a morphism $f:(X,\mathcal F) \lra (Y,\mathcal G)$  is a map 
$f:U_X\lra Y$ defined on a big subset $U_X\in \mathcal F$ 
%of the underlying sets 
such that the preimage of a big set  is big, 
i.e.~$\{ f\inv(U): U\in \mathcal G\}\subset \mathcal F$. 
Two such morphisms are considered identical iff they coincide on a big subset.
We call such maps of filters {\em continuous defined almost everywhere}.
%, as it enables us to say 
Note that we may still say 
that a map of topological spaces is continuous iff the induced maps of neighbourhoods filters
are almost everywhere continuous.

\end{defi}
We may consider a type space $S_n(B)$ to be objects of $\PHi$ if we equip $S_n(B)$
with the following {\em indiscernability} filter generated by sets of types containing 
a formula over $B$ of the form
$$\bigwedge\limits_{0<l<k}( x_{i_l}\neq x_{i_{l+1}} \,\&\, x_{j_l}\neq x_{j_{l+1}}) 
\implies ( \varphi(x_{i_1},...,x_{i_k})\llrra \varphi(x_{j_1},...,x_{j_k}) )$$
Perhaps it is more reasonable to define these filters slightly different 
by taking the formulas of the form, for each $k<N$ and a finite collection of formulas $\varphi_s$ over $B$:
$$\bigwedge\limits_{0<l<k}( x_{i_l}\neq x_{i_{l+1}} \,\&\, x_{j_l}\neq x_{j_{l+1}}) 
\implies 
\exists x_{n+1}...x_N 
(\bigwedge\limits_{n<r<s\leqslant N} x_r\neq x_s
\,\,\,\,\&\,$$ $$
\bigwedge\limits_
				{ \substack{ i_k<i_{k+1}<...<i_r\leqslant  N,\\ 
					\\ j_k<j_{k+1}<...<j_r\leqslant N}}
(\bigwedge\limits_s \varphi_s(x_{i_1},...,x_{i_r})\llrra \varphi_s(x_{j_1},...,x_{j_r}) )$$
The formula is meant to say that the tuple $x_1,..,x_n$ can be extended to an arbitrary long
finite tuple indiscernible with respect to arbitrary finitely many formulas over $B$.

\end{document}